\documentclass[a4paper,11pt]{amsart}

\usepackage{amssymb,mathrsfs,cite,url}

\newtheorem{theorem}{Theorem}[section]
\newtheorem{lemma}[theorem]{Lemma}
\newtheorem{corollary}[theorem]{Corollary}
\newtheorem{proposition}[theorem]{Proposition}

\theoremstyle{definition}
\newtheorem{remark}[theorem]{Remark}
\newtheorem*{xremark}{Remark}

\numberwithin{equation}{section}

\makeatletter
\@namedef{subjclassname@2020}{%
  \textup{2020} Mathematics Subject Classification}
\makeatother

\DeclareMathOperator{\RE}{Re}
\DeclareMathOperator{\IM}{Im}
\DeclareMathOperator{\meas}{meas}

\begin{document}

\title[Large deviations for values of $L$-functions]
{Large deviations for values of $L$-functions attached to cusp forms in the level aspect}

\author[M. Mine]{Masahiro Mine}
\address{Faculty of Science and Technology\\ Sophia University\\ 7-1 Kioi-cho, Chiyoda-ku, Tokyo 102-8554, Japan}
\email{m-mine@sophia.ac.jp}

\date{}

\begin{abstract}
We study the distribution of values of automorphic $L$-functions in a family of holomorphic cusp forms with prime level. 
We prove an asymptotic formula for a certain density function closely related to this value-distribution. 
The formula is applied to estimate large values of $L$-functions. 
\end{abstract}




\maketitle

\section{Introduction}\label{sec:1}
Let $q$ be a prime number. 
Denote by $S_2(q)$ the space of holomorphic cusp forms for the congruence subgroup $\Gamma_0(q)$ of weight $2$ with trivial nebentypus. 
We describe the Fourier series expansion of $f \in S_2(q)$ at infinity as 
\begin{gather*}
f(z)
=\sum_{n=1}^{\infty} a_f(n) \sqrt{n} \exp(2\pi inz)
\end{gather*}
for $z \in \mathbb{C}$ with $\IM(z)>0$ so that the automorphic $L$-function
\begin{gather*}
L(s,f)
=\sum_{n=1}^{\infty} \frac{a_f(n)}{n^s}
\end{gather*}
has the critical strip $0\leq \RE(s) \leq1$. 
The behavior of the values $L(s,f)$ in the critical strip has received attention by many researchers. 
Put $\mathfrak{q}(s)=q(|s|+3)^2$ for $s \in \mathbb{C}$. 
By a standard method of $L$-functions, we have the convexity bound
\begin{gather}\label{eq:02141557}
L(s,f)
\ll_\epsilon \mathfrak{q}(s)^{(1-\sigma)/2+\epsilon}
\end{gather}
for $s=\sigma+it$ with $0\leq \sigma \leq1$, where $\epsilon$ is any positive real number. 
Then the work of reducing the exponent in \eqref{eq:02141557} proceeded, and some results are included in \cite{DukeFriedlanderIwaniec1994, DukeFriedlanderIwaniec2001, BlomerHarcosMichel2007}. 
The Grand Riemann Hypothesis (GRH) is also useful to derive a sharp upper bound of $\log{L}(s,f)$ for $\RE(s)>1/2$. 
Let $B_2(q)$ be a basis of $S_2(q)$ consisting of primitive cusp forms. 
Throughout this paper, $\log_j$ indicates the $j$-fold iterated natural logarithm, that is, 
\begin{gather*}
\log_1=\log
\quad\text{and}\quad
\log_{j+1}=\log(\log_j) 
\end{gather*}
for $j \geq1$. 
Then GRH implies
\begin{gather}\label{eq:02141800}
\log{L}(s,f)
\ll \frac{(\log{\mathfrak{q}(s)})^{2-2\sigma}}{(2\sigma-1)\log_2{\mathfrak{q}(s)}}+\log_2{q(s)}
\end{gather}
for $s=\sigma+it$ with $1/2<\sigma \leq5/4$ if $f $ belongs to $B_2(q)$; see \cite[Theorem 5.19]{IwaniecKowalski2004}. 
It is believed that even \eqref{eq:02141800} is not best possible. 
Moreover, the true order of the magnitude of $\log{L}(\sigma,f)$ is expected to be $(\log{q})^{1-\sigma+o(1)}$ for $1/2<\sigma<1$ by some probabilistic observations. 
Define the distribution function
\begin{gather*}
\Phi_q(\sigma,\tau)
=\frac{\# \left\{f \in B_2(q) ~\middle|~ \log{L}(\sigma,f)>\tau\right\}}{\# B_2(q)} 
\end{gather*}
for $\sigma>1/2$ and $\tau \in \mathbb{R}$. 
Then it is reasonable to consider the decay of $\Phi_q(\sigma,\tau)$ with respect to $\tau$ such that $\tau \approx (\log{q})^{1-\sigma+o(1)}$ for $1/2<\sigma<1$. 
For technical reasons, the weighted distribution function 
\begin{gather*}
\widetilde{\Phi}_q(\sigma,\tau)
=\Bigg(\sum_{f \in B_2(q)} \omega_f\Bigg)^{-1}
\sum_{\substack{f \in B_2(q) \\ \log{L}(\sigma,f)>\tau}} \omega_f
\end{gather*}
is usually easier to deal with, where $\omega_f=(4\pi \langle f,f \rangle)^{-1}$ is the harmonic weight, and $\langle f,g \rangle$ indicates the Petersson inner product on $\Gamma_0(q) \backslash \mathbb{H}$. 
Lamzouri \cite{Lamzouri2011b} proved that there exist positive constants $A(\sigma)$ and $c(\sigma)$ satisfying 
\begin{gather}\label{eq:02151442}
\widetilde{\Phi}_q(\sigma,\tau)
=\exp\left(-A(\sigma) \tau^{\frac{1}{1-\sigma}} (\log{\tau})^{\frac{\sigma}{1-\sigma}}
\left(1+O\left( \frac{1}{\sqrt{\log{\tau}}}+r(\log{q},\tau) \right)\right)\right)
\end{gather}
uniformly in the range $1\ll \tau \leq c(\sigma)(\log{q})^{1-\sigma}(\log_2{q})^{-1}$, where 
\begin{gather*}
r(y,\tau)
=\left(\frac{\tau}{y^{1-\sigma}(\log{y})^{-1}}\right)^{(\sigma-\frac{1}{2})/(1-\sigma)}. 
\end{gather*}
One can deduce from \eqref{eq:02151442} a similar result for the distribution function $\Phi_q(\sigma,\tau)$. 
Recall the following facts on cusp forms:  
\begin{gather*}
\# B_2(q)
= \frac{q}{12}+O(1); 
\qquad
\sum_{f \in B_2(q)} \omega_f 
=1+O\left(q^{-3/2}\log{q}\right); \\
q^{-1}(\log{q})^{-3} 
\ll \omega_f 
\ll q^{-1}\log{q}. 
\end{gather*}
See \cite{Miyake2006} and \cite[$(1.16)$ and $(1.17)$]{CogdellMichel2004}. 
Thus the relation
\begin{gather*}
\log \Phi_q(\sigma,\tau)
=\log \widetilde{\Phi}_q(\sigma,\tau)+O\left(\log_2{q}\right)
\end{gather*}
follows from the definitions of $\Phi_q(\sigma,\tau)$ and $\widetilde{\Phi}_q(\sigma,\tau)$. 
It implies that \eqref{eq:02151442} remains valid for $\Phi_q(\sigma,\tau)$ in a suitable range of $\tau$. 
The purpose of this paper is to provide a more precise formula of $\Phi_q(\sigma,\tau)$ for $1/2<\sigma<1$ in the same range of $\tau$ as in the result of Lamzouri. 
Furthermore, we consider an analogous issue on $\Phi_q(1,\tau)$. 
Note that several authors \cite{Lamzouri2010, LiuRoyerWu2008, Xiao2016} proved asymptotic formulas of $\widetilde{\Phi}_q(1,\tau)$ that look quite different from \eqref{eq:02151442}; see \eqref{eq:02231541} below.

\subsection{Statement of results}\label{sec:1.1}
Let $\sigma>1/2$ and $\tau \in \mathbb{R}$. 
We begin with the limiting distribution function 
\begin{gather}\label{eq:02261944}
\Phi(\sigma,\tau)
=\lim_{q \to\infty} \Phi_q(\sigma,\tau). 
\end{gather}
The author \cite{Mine2020+} showed the existence of a continuous function $\mathcal{M}_\sigma:\mathbb{R}\to \mathbb{R}_{\geq0}$ such that $\Phi(\sigma,\tau)$ satisfies the identity
\begin{gather}\label{eq:05111549}
\Phi(\sigma,\tau)
=\int_{\tau}^{\infty} \mathcal{M}_\sigma(x) \,|dx|, 
\end{gather}
where $|dx|=(2\pi)^{-1/2}dx$. 
The probability density function $\mathcal{M}_\sigma$ is regarded as an analogue of the \emph{$M$-function} of Ihara \cite{Ihara2008} which was originally studied for Dirichlet $L$-functions. 
See also Lebacque--Zykin \cite{LebacqueZykin2018} and Matsumoto--Umegaki \cite{MatsumotoUmegaki2018, MatsumotoUmegaki2019, MatsumotoUmegaki2020} for related work on $M$-functions for automorphic $L$-functions. 
Then, we define 
\begin{gather}\label{eq:02261551}
F_\sigma(\kappa)
=\int_{\mathbb{R}} e^{\kappa x} \mathcal{M}_\sigma(x) \,|dx|
\quad\text{and}\quad
f_\sigma(\kappa)=\log{F}_\sigma(\kappa)
\end{gather}
for $\kappa \in \mathbb{R}$. 
The first main result is the following asymptotic formula for $\mathcal{M}_\sigma$. 

\begin{theorem}\label{thm:1.1}
Let $1/2<\sigma \leq1$. 
For $\tau>0$, we take a real number $\kappa=\kappa(\sigma,\tau)>0$ as the solution of $f'_\sigma(\kappa)=\tau$. 
Then we obtain
\begin{gather}\label{eq:02152104}
\mathcal{M}_\sigma(\tau+x)
=\frac{F_\sigma(\kappa) e^{-\kappa(\tau+x)}}{\sqrt{f''_\sigma(\kappa)}}
\left\{ \exp\left(-\frac{x^2}{2f''_\sigma(\kappa)}\right)
+O\left(\kappa^{-\frac{1}{2\sigma}} \sqrt{\log{\kappa}}\right) \right\}
\end{gather}
uniformly for all $x \in \mathbb{R}$ if $\tau>0$ is large enough, where the implied constant depends only on $\sigma$. 
\end{theorem}

Remark that the main term of \eqref{eq:02152104} dominates the error term only if $|x|$ is small. 
Nevertheless, we see that Theorem \ref{thm:1.1} yields the formula
\begin{align*}
\Phi(\sigma,\tau)
&=\int_{0}^{\infty} \mathcal{M}_\sigma(\tau+x) \,|dx| \\
&=\frac{F_\sigma(\kappa) e^{-\kappa\tau}}{\kappa \sqrt{2\pi f''_\sigma(\kappa)}}
\left\{ 1+O\left(\kappa^{-\frac{1}{2\sigma}} \sqrt{\log{\kappa}}\right) \right\} 
\end{align*}
for $1/2<\sigma \leq1$. 
A similar formula was also proved by Liu--Royer--Wu \cite{LiuRoyerWu2008} for the limiting distribution function $\widetilde{\Phi}(1,\tau):=\lim_{q \to\infty} \widetilde{\Phi}_q(1,\tau)$. 
In Section \ref{sec:3}, we study the asymptotic behaviors of $f_\sigma(\kappa)$ and its derivatives. 
Then we derive the following corollaries of Theorem \ref{thm:1.1}. 
Throughout this paper, $I_\nu(u)$ denotes the modified Bessel function of the first kind of order $\nu$, and we put $g(u)=\log(I_1(2u)/u)$. 

\begin{corollary}\label{cor:1.2}
Let $1/2<\sigma<1$ and $N \in \mathbb{Z}_{\geq1}$. 
For $n=0,\ldots,N-1$, there exist polynomials $A_{n,\sigma}(x)$ of degree at most $n$ with $A_{0,\sigma}(x)\equiv1$ such that the formula 
\begin{gather*}
\Phi(\sigma,\tau)
=\exp\left(-A(\sigma)\tau^{\frac{1}{1-\sigma}} (\log{\tau})^{\frac{\sigma}{1-\sigma}} 
\left\{ \sum_{n=0}^{N-1} \frac{A_{n,\sigma}(\log_2{\tau})}{(\log{\tau})^n}
+O\left( \left(\frac{\log_2{\tau}}{\log{\tau}}\right)^N \right) \right\}\right)
\end{gather*}
holds if $\tau>0$ is large enough. 
Here, the implied constant depends on $\sigma$ and $N$, and $A(\sigma)$ is the positive constant determined as 
\begin{gather*}
A(\sigma)
=(1-\sigma)
\left(\frac{1-\sigma}{\sigma} \int_{0}^{\infty} g(y^{-\sigma}) \,dy \right)^{-\frac{\sigma}{1-\sigma}}. 
\end{gather*}
Furthermore, the polynomial $A_{1,\sigma}(x)$ is obtained as
\begin{gather*}
A_{1,\sigma}(x)
=\frac{\sigma}{1-\sigma}
\left\{x 
-\log \left(\frac{1-\sigma}{\sigma} \mathfrak{a}_0(\sigma)\right) 
+\frac{1-\sigma}{\sigma} \frac{\mathfrak{a}_1(\sigma)}{\mathfrak{a}_0(\sigma)} \right\}, 
\end{gather*}
where $\mathfrak{a}_0(\sigma)$ and $\mathfrak{a}_1(\sigma)$ are the constants represented as
\begin{gather*}
\mathfrak{a}_0(\sigma)
=\int_{0}^{\infty} g(y^{-\sigma}) \,dy 
\quad\text{and}\quad
\mathfrak{a}_1(\sigma)
=\int_{0}^{\infty} g(y^{-\sigma}) \log{y} \,dy. 
\end{gather*}
\end{corollary}

To simplify the statement in the case $\sigma=1$, we put $\tau=2\log{t}+2\gamma$ by using Euler's constant $\gamma=0.577\ldots$. 
Then we have
\begin{gather*}
\Phi_q(1,\tau)
=\frac{\# \left\{f \in B_2(q) ~\middle|~ {L}(1,f)>(e^\gamma t)^2\right\}}{\# B_2(q)}. 
\end{gather*}
For $u>0$, we define the function $g_*(u)$ as
\begin{gather}\label{eq:02240037}
g_*(u)
=
\begin{cases}
g(u)
& \text{if $0<u \leq1$}, 
\\
g(u)-2u
& \text{if $u>1$}. 
\end{cases}
\end{gather}
The following corollary is a variant of \cite[Theorem 1.5]{LiuRoyerWu2008} which was originally stated for the distribution function $\widetilde{\Phi}(1,\tau)$. 

\begin{corollary}\label{cor:1.3}
Let $N \in \mathbb{Z}_{\geq1}$ and put $\tau=2\log{t}+2\gamma$. 
For $n=0,\ldots,N-1$, there exist real numbers $a_n$ with $a_0=1$ such that the formula 
\begin{gather*}
\Phi(1,\tau)
=\exp\left(-\frac{e^{t-A}}{t}
\left\{ \sum_{n=0}^{N-1} \frac{a_n}{t^n}
+O\left(\frac{1}{t^N}\right) \right\}\right)
\end{gather*}
holds if $t>0$ is large enough. 
Here, the implied constant depends on $N$, and $A$ is the constant determined as
\begin{gather*}
A
=1+\frac{1}{2}\int_{0}^{\infty} g_*(y^{-1}) \,dy-\log{2}. 
\end{gather*}
Furthermore, the real number $a_1$ is obtained as
\begin{gather*}
a_1
=-\frac{1}{8}\mathfrak{a}_{0}^2
+\frac{1}{2}\mathfrak{a}_{1}
+\frac{1}{2}, 
\end{gather*}
where $\mathfrak{a}_0$ and $\mathfrak{a}_1$ are the constants represented as
\begin{gather*}
\mathfrak{a}_0
=\int_{0}^{\infty} g_*(y^{-1}) \,dy 
\quad\text{and}\quad
\mathfrak{a}_1
=\int_{0}^{\infty} g_*(y^{-1}) \log{y} \,dy. 
\end{gather*}
\end{corollary}

Then, we proceed to study the distribution function $\Phi_q(\sigma,\tau)$ for $1/2<\sigma \leq1$. 
Using a certain asymptotic formula for complex moments of $L(\sigma,f)$ proved in \cite{Mine2020+}, we associate $\Phi_q(\sigma,\tau)$ with $\Phi(\sigma,\tau)$ as follows. 

\begin{theorem}\label{thm:1.4}
Let $B \geq1$ be a real number. 
\begin{itemize}
\item[$(\mathrm{i})$]
For $1/2<\sigma<1$, there exists a positive constant $c(\sigma,B)$ such that 
\begin{gather*}
\Phi_q(\sigma,\tau)
=\Phi(\sigma,\tau)
\left( 1+O\left(\frac{1}{(\log{q})^B}
+\frac{(\tau \log{\tau})^{\frac{\sigma}{1-\sigma}}}{(\log{q})^\sigma}\right) \right)
\end{gather*}
holds uniformly in the range $1\ll \tau \leq c(\sigma,B)(\log{q})^{1-\sigma}(\log_2{q})^{-1}$, where the implied constant depends on $\sigma$ and $B$. 
\item[$(\mathrm{ii})$]
Put $\tau=2\log{t}+2\gamma$. 
Then there exists a positive constant $c(B)$ such that 
\begin{gather*}
\Phi_q(1,\tau)
=\Phi(1,\tau)
\left( 1+O\left(\frac{1}{(\log{q})^B}
+\frac{e^t}{(\log{q})(\log_2{q}\,\log_3{q})^{-1}}\right) \right)
\end{gather*}
holds uniformly in the range $1\ll t \leq \log_2{q}-\log_3{q}-\log_4{q}-c(B)$, where the implied constant depends on $B$. 
\end{itemize}
\end{theorem}

From Corollary \ref{cor:1.2} and Theorem \ref{thm:1.4} $(\mathrm{i})$, we deduce the following result on $\Phi_q(\sigma,\tau)$ for $1/2<\sigma<1$. 
It refines \eqref{eq:02151442} in the desired range of $\tau$. 

\begin{corollary}\label{cor:1.5}
Let $1/2<\sigma<1$ and $N \in \mathbb{Z}_{\geq1}$. 
Then there exists a positive constant $c(\sigma)$ such that 
\begin{gather*}
\Phi_q(\sigma,\tau)
=\exp\left(-A(\sigma)\tau^{\frac{1}{1-\sigma}} (\log{\tau})^{\frac{\sigma}{1-\sigma}} 
\left\{ \sum_{n=0}^{N-1} \frac{A_{n,\sigma}(\log_2{\tau})}{(\log{\tau})^n}
+O\left( \left(\frac{\log_2{\tau}}{\log{\tau}}\right)^N \right) \right\}\right)
\end{gather*}
holds uniformly in the range $1\ll \tau \leq c(\sigma)(\log{q})^{1-\sigma}(\log_2{q})^{-1}$, where $A(\sigma)$ and $A_{n,\sigma}(x)$ are as in Corollary \ref{cor:1.2}. 
Here, the implied constant depends on $\sigma$ and $N$. 
\end{corollary}

In addition to \eqref{eq:02151442}, Lamzouri \cite{Lamzouri2010} proved the asymptotic formula
\begin{gather}\label{eq:02231541}
\widetilde{\Phi}_q(1,\tau)
=\exp\left(-\frac{e^{t-A}}{t}
\left(1+O\left(\frac{1}{\sqrt{t}}\right)\right) \right)
\end{gather}
uniformly in the range $1\ll t \leq \log_2{q}-\log_3{q}-2\log_4{q}$, where the notation is the same as above. 
Using Corollary \ref{cor:1.3} and Theorem \ref{thm:1.4} $(\mathrm{ii})$, we deduce an asymptotic formula for $\Phi_q(1,\tau)$ in a similar range of $t$. 

\begin{corollary}\label{cor:1.6}
Let $N \in \mathbb{Z}_{\geq1}$ and put $\tau=2\log{t}+2\gamma$. 
Then there exists a positive constant $c$ such that 
\begin{gather*}
\Phi_q(1,\tau)
=\exp\left(-\frac{e^{t-A}}{t}
\left\{ \sum_{n=0}^{N-1} \frac{a_n}{t^n}
+O\left(\frac{1}{t^N}\right) \right\}\right)
\end{gather*}
holds uniformly in the range $1\ll t \leq \log_2{q}-\log_3{q}-\log_4{q}-c$, where $A$ and $a_n$ are as in Corollary \ref{cor:1.3}. 
Here, the implied constant depends on $N$. 
\end{corollary}

\begin{xremark}
\begin{itemize}
\item[$(\mathrm{i})$]
For $\sigma>1/2$ and $\tau \in \mathbb{R}$, we define another distribution function
\begin{gather*}
\Psi_q(\sigma,\tau)
=\frac{\# \left\{f \in B_2(q) ~\middle|~ \log{L}(\sigma,f)<-\tau\right\}}{\# B_2(q)}. 
\end{gather*}
Then the limiting distribution function $\Psi(\sigma,\tau):=\lim_{q \to\infty} \Psi_q(\sigma,\tau)$ satisfies 
\begin{gather*}
\Psi(\sigma,\tau)
=\int_{-\infty}^{-\tau} \mathcal{M}_\sigma(x) \,|dx|
=\int_{-\infty}^{0} \mathcal{M}_\sigma(-\tau+x) \,|dx|. 
\end{gather*}
For $1/2<\sigma \leq1$, one can prove a formula for $\mathcal{M}_\sigma(-\tau+x)$ similar to \eqref{eq:02152104} by replacing $F_\sigma(\kappa)$ and $f_\sigma(\kappa)$ of \eqref{eq:02261551} with $F_\sigma(-\kappa)$ and $f_\sigma(-\kappa)$, respectively. 
We see that Corollaries \ref{cor:1.2} and \ref{cor:1.3} remain true for $\Psi(\sigma,\tau)$. 
Furthermore, the method for the proof of Theorem \ref{thm:1.4} is available to compare $\Psi_q(\sigma,\tau)$ with $\Psi(\sigma,\tau)$. 
As a result, one can prove that $\Psi_q(\sigma,\tau)$ satisfies the same asymptotic formulas as $\Phi_q(\sigma,\tau)$ described in Corollaries \ref{cor:1.5} and \ref{cor:1.6}. 
\item[$(\mathrm{ii})$]
Liu--Royer--Wu also showed a result similar to Theorem \ref{thm:1.4} $(\mathrm{ii})$ in a family of cusp forms of weight $k \geq12$ and level $1$ as $k \to\infty$. 
If we adopt totally the same method for the purpose of comparing $\Phi_q(1,\tau)$ with $\Phi(1,\tau)$, then the admissible range of $t$ is to be obtained as $1\ll t \leq T(q)$, where
\begin{gather*}
T(q)
=\log_2{q}-\frac{5}{2}\log_3{q}-\log_4{q}-c
\end{gather*}
with a constant $c>0$; see \cite[Theorem 2]{LiuRoyerWu2008}. 
However, this is narrower than Lamzouri's range $1\ll t \leq \log_2{q}-\log_3{q}-2\log_4{q}$. 
In this paper, we present a modified method of comparing $\Phi_q(\sigma,\tau)$ with $\Phi(\sigma,\tau)$ for $1/2<\sigma \leq1$ so as to fill this gap. 
\end{itemize}
\end{xremark}

\subsection{Related results for other zeta and $L$-functions}\label{sec:1.2}
The value-distributions of zeta and $L$-functions of degree one are classical topics in analytic number theory. 
For the Riemann zeta-function $\zeta(s)$, we define the distribution function
\begin{gather*}
\Phi_{1,T}(\sigma,\tau)
=\frac{1}{T} \meas \left\{t \in [0,T] ~\middle|~ \log{|\zeta(\sigma+it)|}>\tau\right\}, 
\end{gather*}
where $\meas(S)$ is the Lebesgue measure of a set $S \subset \mathbb{R}$. 
For the Dirichlet $L$-function attached to the quadratic character $\chi_d(n)=(\frac{d}{n})$, we also define 
\begin{gather*}
\Phi_{2,x}(\sigma,\tau)
=\Bigg(\sideset{}{^\flat}\sum_{|d|\leq x} 1\Bigg)^{-1} 
\sideset{}{^\flat}\sum_{\substack{|d|\leq x \\ \log{L}(\sigma,\chi_d)>\tau}} 1, 
\end{gather*}
where $\sideset{}{^\flat}\sum$ indicates the sum over fundamental discriminants. 
We explore several results of $\Phi_{1,T}(\sigma,\tau)$ and $\Phi_{2,x}(\sigma,\tau)$ for comparisons with the results described in Section \ref{sec:1.1}. 
First, we note that there exist the limiting distribution functions 
\begin{gather*}
\Phi_1(\sigma,\tau)
=\lim_{T \to\infty} \Phi_{1,T}(\sigma,\tau) 
\quad\text{and}\quad
\Phi_2(\sigma,\tau)
=\lim_{x \to\infty} \Phi_{2,x}(\sigma,\tau) 
\end{gather*}
for $\sigma>1/2$ and $\tau \in \mathbb{R}$, which were essentially proved by Bohr--Jessen \cite{BohrJessen1930, BohrJessen1932} for $\Phi_1(\sigma,\tau)$ and by Chowla--Erd\H{o}s \cite{ChowlaErdos1951} for $\Phi_2(\sigma,\tau)$. 
The estimates of these distribution functions were improved along with the work of applying methods of probability theory to problems of number theory. 
In particular, Granville--Soundararajan \cite{GranvilleSoundararajan2003, GranvilleSoundararajan2006} applied the saddle-point method to derive the formulas
\begin{gather}\label{eq:02241049}
\Phi_j(1,\tau)
=\exp\left(-\frac{e^{t-A_j}}{t}
\left(1+O\left(\frac{1}{\sqrt{t}}\right)\right) \right)
\end{gather}
for $j=1,2$, where we put $\tau=\log{t}+\gamma$, and $A_j$ are the constants determined as follows. 
Define $g_1(u)=\log{I}_0(u)$ and $g_2(u)=\log\cosh(u)$ for $u>0$, and put 
\begin{gather*}
g_{j,*}(u)
=
\begin{cases}
g_j(u)
& \text{if $0<u \leq1$}, 
\\
g_j(u)-u
& \text{if $u>1$} 
\end{cases}
\end{gather*}
similarly to \eqref{eq:02240037}. 
By these functions, the constants $A_j$ are represented as
\begin{gather*}
A_j
=1+\int_{0}^{\infty} g_{j,*}(y^{-1}) \,dy. 
\end{gather*}
Then, Wu \cite{Wu2007} improved formula \eqref{eq:02241049} in the form
\begin{gather*}
\Phi_2(1,\tau)
=\exp\left(-\frac{e^{t-A_2}}{t}
\left\{\sum_{n=0}^{N-1} \frac{\alpha_n}{t^n}
+O\left(\frac{1}{t^N}\right)\right\} \right), 
\end{gather*}
where $\alpha_0,\ldots,\alpha_{N-1}$ are real numbers such that $\alpha_0=1$. 
Wu's method was also based on the saddle-point method, but the treatment of the saddle-point was slightly different from Granville--Soundararajan's one. 
We prove Theorem \ref{thm:1.1} by modifying the method of \cite{Wu2007} rather than \cite{GranvilleSoundararajan2003, GranvilleSoundararajan2006}. 
Furthermore, Theorem \ref{thm:1.4} is regarded as an analogue of the formula
\begin{gather}\label{eq:02241159}
\Phi_{2,x}(1,\tau)
=\Phi_2(1,\tau)
\left(1+O\left(\frac{1}{(\log{x})^5}
+\frac{e^t}{(\log{x}\log_3{x})(\log_2{x})^{-2}}\right)\right), 
\end{gather}
which was shown in \cite{GranvilleSoundararajan2003} uniformly in the range $1\ll t \leq \log_2{x}-2\log_3{x}+\log_4{q}-20$. 
Asymptotic behaviors of $\Phi_j(\sigma,\tau)$ for $1/2<\sigma<1$ were studied by Lamzouri \cite{Lamzouri2011b}. 
Let $A_j(\sigma)$ be the positive constants determined by 
\begin{gather*}
A_j(\sigma)
=(1-\sigma)\left(\frac{1-\sigma}{\sigma}
\int_{0}^{\infty} g_j(y^{-\sigma}) \,dy\right)^{-\frac{\sigma}{1-\sigma}}. 
\end{gather*}
Then he proved the asymptotic formulas
\begin{gather*}
\Phi_j(\sigma,\tau)
=\exp\left(-A_j(\sigma) \tau^{\frac{1}{1-\sigma}} (\log{\tau})^{\frac{\sigma}{1-\sigma}}
\left(1+O\left(\frac{1}{\sqrt{\log{\tau}}}\right) \right)\right)
\end{gather*}
for $j=1,2$. 
Note that a similar result for $j=1$ was seen in the earlier work of Hattori--Matsumoto \cite{HattoriMatsumoto1999}. 
Finally, an analogue of \eqref{eq:02241159} for $\Phi_{1,T}(\sigma,\tau)$ was achieved by Lamzouri--Lester--Radziwi{\l\l} \cite[Theorem 1.3]{LamzouriLesterRadziwill2019}. 
See also \cite{EndoInoueMine2020+} for a refinement.

\subsection*{Organization of the paper}
This paper consists of five sections. 
\begin{itemize}
\item
Section \ref{sec:2} is devoted to show some lemmas on certain functions $G_p(z)$ and $G(z)$ defined by the $p$-adic Plancherel measure and the Sato--Tate measure. 
\item
In Section \ref{sec:3}, we prove asymptotic formulas for the function $f_\sigma(\kappa)$ of \eqref{eq:02261551}. 
The main ingredient is to sum up the local components $f_{\sigma,p}(\kappa)$ which are approximated by using the functions $G_p(z)$ and $G(z)$ of Section \ref{sec:2}. 
\item
Section \ref{sec:4} is further divided into three subsections. 
In Section \ref{sec:4.1}, we make preparations for the saddle-point method. 
Then we complete the proof of Theorem \ref{thm:1.1} in Section \ref{sec:4.2}. 
After that, we show corollaries of Theorem \ref{thm:1.1} in Section \ref{sec:4.3}. 
\item
In Section \ref{sec:5}, we present a method of comparing $\Phi_q(\sigma,\tau)$ with $\Phi(\sigma,\tau)$, which is based on the Esseen inequality of probability theory. 
We finally prove Theorem \ref{thm:1.4} and its corollaries to end this paper. 
\end{itemize}

\subsection*{Acknowledgment}
The author would like to thank Kenta Endo and Sh\={o}ta Inoue for helpful discussions. 
The work of this paper was supported by Grant-in-Aid for JSPS Fellows (Grant Number JP21J00529).

\section{Preliminary lemmas}\label{sec:2}
Let $p$ be a prime number. 
Denote by $\mu_p$ the $p$-adic Plancherel measure on the interval $[0,\pi]$ defined as 
\begin{gather}\label{eq:02271312}
d \mu_p(\theta)
=\left(1+\frac{1}{p}\right) \left(1-\frac{2\cos2\theta}{p}+\frac{1}{p^2}\right)^{-1}
\frac{2}{\pi} \sin^2\theta \,d \theta. 
\end{gather}
Then $\mu_p$ converges weakly to the Sato--Tate measure $\mu_\infty$ as $p \to\infty$. 
In this section, we show some preliminary lemmas on the functions 
\begin{gather}\label{eq:02251719}
G_p(z)
=\int_{0}^{\pi} \exp(2z \cos \theta) \,d \mu_p(\theta) 
\quad\text{and}\quad
G(z)
=\int_{0}^{\pi} \exp(2z \cos \theta) \,d \mu_\infty(\theta) 
\end{gather}
defined for $z=u+iv \in \mathbb{C}$. 
Note that $G(z)$ is an even entire function represented as 
\begin{align}\label{eq:03012035}
G(z)
&=\frac{1}{\pi} \int_{0}^{\pi} \exp(2z \cos \theta) (1-\cos2\theta) \,d \theta \\
&=I_0(2z)-I_2(2z)
=\frac{I_1(2z)}{z} \nonumber
\end{align}
by definition. 
Thus, we deduce the following properties of $G(z)$ from the results of the Bessel functions; see \cite{Watson1995}. 

\begin{lemma}\label{lem:2.1}
\begin{itemize}
\item[$(\mathrm{i})$]
All zeros of $G(z)$ lie on the imaginary axis. 
In particular, the ordinate of the first zero on the upper half-plane is $7.66\ldots$. 
\item[$(\mathrm{ii})$]
Let $z \in \mathbb{C}$ with $\RE(z)>0$. 
Then we have
\begin{gather*}
G(z)
=\frac{1}{\sqrt{4\pi}} \frac{\exp(2z)}{z^{3/2}}
\left(1+O\left(|z|^{-1}\right)\right)
\end{gather*}
as $|z|\to \infty$. 
Here, the branch of $z^{3/2}$ is chosen so that it is real-valued on the positive real axis. 
\item[$(\mathrm{iii})$]
We have $G(z)=1+\frac{1}{2}z^2+O(|z|^4)$ as $|z|\to0$. 
\end{itemize}
\end{lemma}

With regard to the function $G_p(z)$, there seems to be no simple representations by the Bessel functions such as \eqref{eq:03012035}. 
On the other hand, we know that $G_p(z) \to G(z)$ uniformly as $p \to\infty$. 
More precisely, we prove the following result. 

\begin{lemma}\label{lem:2.2}
Let $z=u+iv$ with $u>0$ and $|v|\leq u$. 
Then we have uniformly
\begin{gather*}
G_p(z)
=G(z)
\left(1+O\left(p^{-1}\right)\right)
\end{gather*}
for every prime number $p$.  
\end{lemma}

\begin{proof}
Since we have
\begin{gather*}
\left(1+\frac{1}{p}\right) \left(1-\frac{2\cos2\theta}{p}+\frac{1}{p^2}\right)^{-1}
=1+O(p^{-1})
\end{gather*}
for all $\theta \in [0,\pi]$, the difference between $G_p(z)$ and $G(z)$ is evaluated as 
\begin{gather*}
G_p(z)-G(z)
\ll p^{-1}
\int_{0}^{\pi} \exp(2u \cos \theta) \frac{2}{\pi} \sin^2\theta \,d \theta
=p^{-1}G(u). 
\end{gather*}
By Lemma \ref{lem:2.1} $(\mathrm{ii})$, we have $|G(z)| \asymp G(u)$ for $z=u+iv$ with $u>0$ and $|v|\leq u$. 
Hence we obtain 
\begin{gather*}
\frac{G_p(z)-G(z)}{G(z)}
\ll p^{-1}, 
\end{gather*}
which yields the desired formula. 
\end{proof}

By Lemma \ref{lem:2.1} $(\mathrm{i})$, one can define $g(z)=\log{G}(z)$ as a holomorphic function on the right half-plane $\RE(z)>0$, where the branch of the logarithm is taken so that $g(u)\in \mathbb{R}$ for $u>0$. 
Then we deduce the following estimates of $g(u)$ and its derivatives from Lemma \ref{lem:2.1} $(\mathrm{ii}), (\mathrm{iii})$. 
\begin{lemma}\label{lem:2.3}
We have 
\begin{align*}
g(u)
&=
\begin{cases}
O(u^2) & \text{if $0<u \leq1$}, 
\\
2u+O(\log{u}) & \text{if $u>1$}, 
\end{cases} \\
g'(u)
&=
\begin{cases}
O(u) & \text{if $0<u \leq1$}, 
\\
2+O(u^{-1}) & \text{if $u>1$}, 
\end{cases}
\end{align*}
and for all $j \geq2$, 
\begin{gather*}
g^{(j)}(u)
=
\begin{cases}
O(j!) & \text{if $0<u \leq1$}, 
\\
O(2^jj!u^{-j}) & \text{if $u>1$}.  
\end{cases}
\end{gather*}
\end{lemma}

\begin{proof}
Let $z \in \mathbb{C}$ with $\RE(z)>1/2$. 
Using Lemma \ref{lem:2.1} $(\mathrm{ii})$, we derive the formula
\begin{gather*}
g(z)
=2z-\frac{3}{2}\log{z}-\frac{1}{2}\log4\pi+h(z),  
\end{gather*}
where $h(z)$ is a holomorphic function such that $h(z)\ll |z|^{-1}$. 
It yields the desired estimate of $g(u)$ for $u>1$. 
We also obtain
\begin{gather*}
g^{(j)}(z)
=2\delta_{j,1}
+(j-1)!\frac{3(-1)^j}{2z^j}
+h^{(j)}(z)
\end{gather*}
for all $j \geq1$, where $\delta_{j,1}$ equals to 1 if $j=1$ and 0 otherwise. 
For $u>1$, Cauchy's integral formula yields
\begin{gather*}
h^{(j)}(u)
=\frac{j!}{2\pi i} \int_{|z-u|=u/2} \frac{h(z)}{(z-u)^{j+1}} \,dz 
\ll 2^jj!u^{-j-1}, 
\end{gather*}
and therefore the results on the derivatives follow. 
The results for $0<u \leq1$ can be proved similarly. 
Indeed, we deduce from Lemma \ref{lem:2.1} $(\mathrm{iii})$ that $g(z)=\frac{1}{2}z^2+O(|z|^4)$ for $|z|\leq1$. 
Then the upper bounds of $g(u)$ and $g'(u)$ follow immediately. 
For $j \geq2$, we again use Cauchy's integral formula to obtain 
\begin{gather*}
g^{(j)}(u)
=\frac{j!}{2\pi i} \int_{|z-u|=1} \frac{g(z)}{(z-u)^{j+1}} \,dz
\ll j!
\end{gather*}
by noting that $g(z)$ is bounded on the disk $|z|\leq2$. 
\end{proof}

Let $n$ and $j$ be non-negative integers. 
If $1/2<\sigma<1$, then the integral 
\begin{gather}\label{eq:02261606}
\mathfrak{g}_{n,j}(\sigma)
=\int_{0}^{\infty} \frac{g^{(j)}(u)}{u^{\frac{1}{\sigma}+1-j}} (\log{u})^n \,du
\end{gather}
is finite by Lemma \ref{lem:2.3}. 
If $\sigma=1$, we modify the function $g(u)$ by $g_*(u)$ as in \eqref{eq:02240037}. 
Then we again deduce from Lemma \ref{lem:2.3} that the integral
\begin{gather}\label{eq:02261607}
\mathfrak{g}_{n,j}
=\int_{0}^{\infty} \frac{g_*^{(j)}(u)}{u^{2-j}} (\log{u})^n \,du
\end{gather}
is finite. 
Moreover, it can be easily check that $\mathfrak{g}_{0,j}(\sigma) \ll j!$ and $\mathfrak{g}_{0,j} \ll j!$ by the integrations by parts. 
The constants $\mathfrak{a}_0(\sigma)$, $\mathfrak{a}_1(\sigma)$, $\mathfrak{a}_0$, $\mathfrak{a}_1$ of Corollaries \ref{cor:1.2} and \ref{cor:1.3} are related to these integrals when $n=0,1$.

\section{Estimates of cumulant-generating functions}\label{sec:3}
Let $\Theta=(\Theta_p)_p$ be a sequence of independent random variables distributed on the interval $[0,\pi]$ according to the measure $\mu_p$ of \eqref{eq:02271312}. 
Then we define
\begin{gather}\label{eq:02261940}
L(\sigma, \Theta)
=\prod_{p} \left(1-2(\cos \Theta_p)p^{-\sigma}+p^{-2\sigma}\right)^{-1}. 
\end{gather}
We see that $L(\sigma, \Theta)$ presents an $\mathbb{R}$-valued random variable for $\sigma>1/2$ since the right-hand side of \eqref{eq:02261940} converges almost surely. 
For $\theta \in [0,\pi]$, we also define 
\begin{gather}\label{eq:02262144}
\lambda_{p,\sigma}(\theta)
=\sum_{m=1}^{\infty} \frac{\cos(m \theta)}{m} p^{-m \sigma}. 
\end{gather}
Then we have $\log{L}(\sigma,\Theta)=\sum_{p} 2\lambda_{p,\sigma}(\Theta_p)$ for $\sigma>1/2$. 
The random Euler product $L(\sigma, \Theta)$ is associated with the value-distribution of $L(\sigma,f)$. 
Indeed, the limiting distribution function $\Phi(\sigma,\tau)$ of \eqref{eq:02261944} is represented as 
\begin{gather*}
\Phi(\sigma,\tau)
=\mathbb{P}\left(\log{L}(\sigma,\Theta)>\tau\right) 
\end{gather*}
for $\sigma>1/2$ and $\tau \in \mathbb{R}$, where $\mathbb{P}(E)$ denotes the probability of an event $E$. 
Thus, the $M$-function $\mathcal{M}_\sigma$ is a probability density function of the random variable $\log{L}(\sigma,\Theta)$. 
It was proved in \cite{Mine2020+} that the moment-generating function 
\begin{gather*}
F_\sigma(s)
=\mathbb{E}\left[\exp\left(s \log{L}(\sigma,\Theta)\right)\right]
=\int_{\mathbb{R}} e^{sx} \mathcal{M}_\sigma(x) \,|dx|
\end{gather*}
is defined for all $s=u+iv \in \mathbb{C}$ and has the infinite product representation
\begin{gather}\label{eq:02262021}
F_\sigma(s)
=\prod_{p} F_{\sigma,p}(s), 
\end{gather}
where $F_{\sigma,p}(s)=\mathbb{E}\left[\exp\left(2s \lambda_{p,\sigma}(\Theta_p) \right)\right]$. 
The goal of this section is to show the following formulas of the cumulant-generating function $f_\sigma(\kappa)=\log{F}_\sigma(\kappa)$ and its derivatives.  

\begin{proposition}\label{prop:3.1}
Let $1/2<\sigma<1$ and $N \in \mathbb{Z}_{\geq1}$. 
For all $j \geq0$, we obtain
\begin{gather*}
f_\sigma^{(j)}(\kappa)
=\frac{\kappa^{\frac{1}{\sigma}-j}}{\log{\kappa}}
\left\{\sum_{n=0}^{N-1} \frac{\mathfrak{g}_{n,j}(\sigma)}{(\log{\kappa})^n}
+O\left( \frac{2^jj!}{(\log{\kappa})^N} \right)\right\}
\end{gather*}
if $\kappa>0$ is large enough. 
Here, $\mathfrak{g}_{n,j}(\sigma)$ are defined as \eqref{eq:02261606}, and the implied constants depend on $\sigma$ and $N$.  
\end{proposition}

\begin{proposition}\label{prop:3.2}
Let $N \in \mathbb{Z}_{\geq1}$. 
Then we obtain
\begin{gather*}
f_1(\kappa)
=2\kappa(\log_2{\kappa}+\gamma)
+\frac{\kappa}{\log{\kappa}}
\left\{\sum_{n=0}^{N-1} \frac{\mathfrak{g}_{n,0}}{(\log{\kappa})^n}
+O\left( \frac{1}{(\log{\kappa})^N} \right)\right\}, \\
f'_1(\kappa)
=2(\log_2{\kappa}+\gamma)
+\frac{1}{\log{\kappa}}
\left\{\sum_{n=0}^{N-1} \frac{\mathfrak{g}_{n,1}}{(\log{\kappa})^n}
+O\left( \frac{1}{(\log{\kappa})^N} \right)\right\}, 
\end{gather*}
and for all $j \geq2$, 
\begin{gather*}
f_1^{(j)}(\kappa)
=\frac{\kappa^{1-j}}{\log{\kappa}}
\left\{\sum_{n=0}^{N-1} \frac{\mathfrak{g}_{n,j}}{(\log{\kappa})^n}
+O\left( \frac{2^jj!}{(\log{\kappa})^N} \right)\right\}
\end{gather*}
if $\kappa>0$ is large enough. 
Here, $\mathfrak{g}_{n,j}$ are defined as \eqref{eq:02261607}, and the implied constants depend on $N$. 
\end{proposition}

Let $1/2<\sigma \leq1$ and $N \in \mathbb{Z}_{\geq1}$. 
For $\kappa \geq 6$, we determine the parameters $y_1, y_2$ such that $2\leq y_1<y_2$ by the equations
\begin{gather*}
\kappa y_1^{-2\sigma}
=\delta
\quad\text{and}\quad
\kappa y_2^{-\sigma}
=(\log{\kappa}\,\log_2{\kappa})^{-\sigma N/(2\sigma-1)}, 
\end{gather*}
where $\delta>0$ is a small absolute constant chosen later. 
Then we show three lemmas on the local factors $F_{\sigma,p}(s)$ of \eqref{eq:02262021} toward the proofs of Propositions \ref{prop:3.1} and \ref{prop:3.2}. 

\begin{lemma}\label{lem:3.3}
Let $1/2<\sigma \leq1$ and $\kappa \geq6$. 
For any $p \leq y_1$, we have 
\begin{gather*}
F_{\sigma,p}(s)
=\frac{1}{\sqrt{4\pi}}
\frac{\exp(2s \lambda_{p,\sigma}(0))}{(s|\lambda''_{p,\sigma}(0)|)^{3/2}}
(1+p^{-1})(1-p^{-1})^{-2}
\left(1+O\left(\frac{1}{\sqrt{\kappa}}\right)\right) 
\end{gather*}
uniformly in the disk $|s-\kappa|\leq \kappa/2$, where $\lambda_{p,\sigma}(\theta)$ is defined as \eqref{eq:02262144}, and the implied constant depends only on the choice of $\delta$.  
\end{lemma}

\begin{proof}
Let $\Lambda_p(\theta)=(1+p^{-1}) (1-2(\cos \theta)p^{-1}+p^{-2})^{-1}$. 
We represent $F_{\sigma,p}(s)$ as 
\begin{gather*}
F_{\sigma,p}(s)
=\frac{2}{\pi}
\int_{0}^{\pi} \exp(2s \lambda_{p,\sigma}(\theta)) 
\Lambda_p(2\theta)
\sin^2\theta \,d \theta.
\end{gather*}
Then we estimate the integral by noting that $\lambda_{p,\sigma}(\theta)$ takes the maximum value at $\theta=0$ in the interval $[0,\pi]$. 
Recall that $\lambda'_{p,\sigma}(0)=0$ and $\lambda''_{p,\sigma}(0)<0$. 
By the Taylor series expansion, we have 
\begin{gather*}
\lambda_{p,\sigma}(\theta)
=\lambda_{p,\sigma}(0)
+\frac{\lambda''_{p,\sigma}(0)}{2}\theta^2 \mu_{p,\sigma}(\theta), 
\end{gather*}
where $\mu_{p,\sigma}(\theta)$ is a function of $\theta$ represented as
\begin{gather}\label{eq:02262229}
\mu_{p,\sigma}(\theta)
=1+\sum_{j=1}^{\infty} \frac{2}{\lambda''_{p,\sigma}(0)} \frac{\lambda^{(j+2)}_{p,\sigma}(0)}{(j+2)!} \theta^j. 
\end{gather}
The derivatives of $\lambda_{p,\sigma}(\theta)$ at $\theta=0$ are evaluated as 
\begin{gather*}
\lambda''_{p,\sigma}(0)
=-\sum_{m=1}^{\infty} m p^{-m \sigma}
\asymp p^{-\sigma}
\quad\text{and}\quad
\lambda^{(k)}_{p,\sigma}(0)
\ll k!p^{-\sigma}
\end{gather*}
for all $k \geq3$, where the implied constants are absolute. 
Thus the coefficients of the power series \eqref{eq:02262229} are uniformly bounded, and moreover, there exists a small absolute constant $c>0$ such that $\mu_{p,\sigma}(\theta)=1+O(|\theta|)$ holds uniformly for $|\theta|\leq c$. 
Then, we consider the integral 
\begin{align*}
I_1
&=\int_{0}^{c} \exp\left(2s \left(\lambda_{p,\sigma}(\theta)-\lambda_{p,\sigma}(0) \right)\right)  
\Lambda_p(2\theta)
\sin^2\theta \,d \theta \\
&=\int_{0}^{c} \exp\left(s \lambda''_{p,\sigma}(0) \theta^2 \mu_{p,\sigma}(\theta)\right) 
\Lambda_p(2\theta)
\sin^2\theta \,d \theta. 
\end{align*}
We make a change of variables such that $\phi=\theta^2 \mu_{p,\sigma}(\theta)$. 
For $|\theta|\leq c$, we have
\begin{align*}
\Lambda_p(2\theta)
&=\Lambda_p(0)\left(1+O(\phi)\right), \\
\sin^2\theta
&=\phi\left(1+O(\sqrt{\phi})\right), \\
\frac{d \theta}{d \phi}
&=\frac{1}{2\sqrt{\phi}}\left(1+O(\sqrt{\phi})\right)
\end{align*}
with absolute implied constants. 
Hence we obtain
\begin{gather*}
I_1
=\frac{1}{2} \Lambda_p(0)
\int_{0}^{c_1} \exp\left(s \lambda''_{p,\sigma}(0) \phi \right)
\sqrt{\phi} \left(1+O(\sqrt{\phi})\right) \,d \phi, 
\end{gather*}
where we put $c_1=c^2 \mu_{p,\sigma}(c)$. 
To derive the main term, we see that
\begin{align*}
\int_{0}^{\infty} \exp\left(s \lambda''_{p,\sigma}(0) \phi \right)
\sqrt{\phi} \,d \phi
&=\int_{0}^{\infty} \exp\left(-s |\lambda''_{p,\sigma}(0)| \phi \right)
\sqrt{\phi} \,d \phi \\
&=\frac{\sqrt{\pi}}{2} \frac{1}{(s|\lambda''_{p,\sigma}(0)|)^{3/2}}
\end{align*}
by noting that $\lambda''_{p,\sigma}(0)<0$ and $\RE(s)>0$. 
Therefore, $I_1$ is calculated as 
\begin{gather}\label{eq:02270036}
I_1
=\frac{1}{2} \Lambda_p(0)
\left(\frac{\sqrt{\pi}}{2} \frac{1}{(s|\lambda''_{p,\sigma}(0)|)^{3/2}}
-I_{1,1}+I_{1,2}\right), 
\end{gather}
where $I_{1,1}$ and $I_{1,2}$ are the following integrals: 
\begin{align*}
I_{1,1}
&=\int_{c_1}^{\infty} \exp\left(-s |\lambda''_{p,\sigma}(0)| \phi \right)
\sqrt{\phi} \,d \phi; \\
I_{1,2}
&=\int_{0}^{c_1} \exp\left(-s |\lambda''_{p,\sigma}(0)| \phi \right)
O(\phi) \,d \phi. 
\end{align*}
We recall that $\lambda''_{p,\sigma}(0) \asymp p^{-\sigma}$ and $c_1=\mu_{p,\sigma}(c)c^2 \asymp1$ for every $p$. 
Write $s=u+iv$. 
Then $u \asymp \kappa$ holds in the disk $|s-\kappa|\leq \kappa/2$. 
Thus, we evaluate these integrals as
\begin{gather*}
I_{1,1}
\ll \exp\left(-\frac{u|\lambda''_{p,\sigma}(0)|}{2}c_1 \right)
\frac{1}{(u|\lambda''_{p,\sigma}(0)|)^{3/2}}
\ll (\kappa p^{-\sigma})^{-2}, \\
I_{1,2}
\ll
\int_{0}^{\infty} \exp\left(-u |\lambda''_{p,\sigma}(0)| \phi \right)
\phi \,d \phi
\ll (\kappa p^{-\sigma})^{-2}.
\end{gather*}
Note further that $|s|\asymp \kappa$ in the disk $|s-\kappa|\leq \kappa/2$. 
Hence formula \eqref{eq:02270036} yields
\begin{align}\label{eq:02271539}
I_1
&=\frac{\sqrt{\pi}}{4} \frac{\Lambda_p(0)}{(s|\lambda''_{p,\sigma}(0)|)^{3/2}}
\left(1+O\left(\frac{(|s| p^{-\sigma})^{3/2}}{(\kappa p^{-\sigma})^{2}}\right)\right) \\
&=\frac{\sqrt{\pi}}{4} \frac{\Lambda_p(0)}{(s|\lambda''_{p,\sigma}(0)|)^{3/2}}
\left(1+O\left(\frac{1}{\sqrt{\kappa}}\right)\right) \nonumber
\end{align}
since $\kappa p^{-\sigma} \gg \sqrt{\kappa}$ for any $p \leq y_1$. 
The remaining work is to estimate the integral
\begin{gather*}
I_2
=\int_{c}^{\pi} \exp\left(2s \left(\lambda_{p,\sigma}(\theta)-\lambda_{p,\sigma}(0) \right)\right)
\Lambda_{p}(2\theta)
\sin^2\theta \,d \theta. 
\end{gather*}
The function $\lambda_{p,\sigma}(\theta)$ is decreasing on the interval $[0,\pi]$. 
Therefore we obtain 
\begin{align*}
I_2
&\ll \exp\left(2u \left(\lambda_{p,\sigma}(c)-\lambda_{p,\sigma}(0) \right)\right)
\int_{0}^{c} |\Lambda_{p}(2\theta)\sin^2\theta| \,d \theta \\
&\ll \exp\left(-u |\lambda''_{p,\sigma}(0)| c^2 \mu_{p,\sigma}(c)\right). 
\end{align*}
Since $u \asymp \kappa$, $c^2\mu_{p,\sigma}(c) \asymp1$, and $\lambda''_{p,\sigma}(0) \asymp p^{-\sigma}$, we have $I_2 \ll (\kappa p^{-\sigma})^{-2}$. 
It yields
\begin{gather}\label{eq:02271540}
I_2
= \frac{\Lambda_{p}(0)}{(s|\lambda''_{p,\sigma}(0)|)^{3/2}}
O\left(\frac{1}{\sqrt{\kappa}}\right) 
\end{gather}
due to $\Lambda_{p}(0)>1$ for every $p$.
Combining \eqref{eq:02271539} and \eqref{eq:02271540}, we conclude 
\begin{align*}
F_{\sigma,p}(s)
&=\frac{2}{\pi} 
\exp(2s \lambda_{p,\sigma}(0) )
(I_1+I_2) \\
&=\frac{1}{\sqrt{4\pi}}
\frac{\exp(2s \lambda_{p,\sigma}(0))}{(s|\lambda''_{p,\sigma}(0)|)^{3/2}}
\Lambda_{p}(0)
\left(1+O\left(\frac{1}{\sqrt{\kappa}}\right)\right)
\end{align*}
as desired. 
\end{proof}

\begin{lemma}\label{lem:3.4}
Let $1/2<\sigma \leq1$ and $\kappa \geq6$. 
For any $p>y_1$, we have 
\begin{gather*}
F_{\sigma,p}(s)
=G(sp^{-\sigma})
\left(1+O\left(\kappa p^{-2\sigma}+p^{-1}\right)\right) 
\end{gather*}
uniformly in the disk $|s-\kappa|\leq \kappa/2$, where $G(z)$ is defined as \eqref{eq:02251719}. 
\end{lemma}

\begin{proof}
Using the formula $\exp(z)=1+O(|z|)$ with $|z| \leq1$, we have
\begin{align*}
\exp(2s \lambda_{p,\sigma}(\theta))
&=\exp\left(2s(\cos \theta)p^{-\sigma}+O\left(|s| p^{-2\sigma}\right)\right) \\
&=\exp\left(2s(\cos \theta)p^{-\sigma}\right)
\left(1+O\left(\kappa p^{-2\sigma}\right)\right) 
\end{align*}
since $|s| p^{-2\sigma} \ll \kappa p^{-2\sigma} \leq \delta$ is valid in the disk $|s-\kappa|\leq \kappa/2$ for any $p>y_1$. 
Hence $F_{\sigma,p}(s)$ is estimated as 
\begin{align*}
F_{\sigma,p}(s)
&=\int_{0}^{\pi} \exp\left(2s(\cos \theta)p^{-\sigma}\right)
\left(1+O\left(\kappa p^{-2\sigma}\right)\right) \,d \mu_p(\theta) \\
&=G_p(sp^{-\sigma})+O\left(\kappa p^{-2\sigma}G_p(u p^{-\sigma})\right), 
\end{align*}
where $G_p(z)$ is defined as \eqref{eq:02251719}, and we write $s=u+iv$. 
By Lemma \ref{lem:2.2}, we obtain
\begin{gather*}
F_{\sigma,p}(s)
=G(sp^{-\sigma})
\left(1+O\left(p^{-1}+\kappa p^{-2\sigma}\frac{G(u p^{-\sigma})}{|G(s p^{-\sigma})|} \right)\right). 
\end{gather*}
Recall that the estimate $|G(sp^{-\sigma})| \asymp G(u p^{-\sigma})$ holds in the disk $|s-\kappa|\leq \kappa/2$ by Lemma \ref{lem:2.1} $(\mathrm{ii})$. 
Thus the result follows. 
\end{proof}

\begin{lemma}\label{lem:3.5}
Let $1/2<\sigma \leq1$ and $\kappa \geq6$. 
For any $p \geq y_2$, we have 
\begin{gather*}
F_{\sigma,p}(s)
=1+O\left(\kappa^2 p^{-2\sigma}\right) 
\end{gather*}
uniformly in the disk $|s-\kappa|\leq \kappa/2$. 
\end{lemma}

\begin{proof}
By the Taylor series expansion, we have
\begin{gather*}
\exp(s \lambda_{p,\sigma}(\theta))
=1+2s(\cos \theta)p^{-\sigma}+O\left(\kappa^2 p^{-2\sigma}\right)
\end{gather*}
since $|s| p^{-\sigma} \ll \kappa p^{-\sigma} \leq 1$ is valid in the disk $|s-\kappa|\leq \kappa/2$ for any $p \geq y_2$. 
By simple calculations, we find the equalities
\begin{gather*}
\int_{0}^{\pi} \,d \mu_p(\theta)=1
\quad\text{and}\quad
\int_{0}^{\pi} (\cos \theta) \,d \mu_p(\theta)=0. 
\end{gather*}
Hence we obtain the conclusion. 
\end{proof}

Let $s=u+iv$ with $|s-\kappa|\leq \kappa/2$. 
If $\kappa>0$ is large enough, then we deduce from Lemma \ref{lem:3.3} that $F_{\sigma,p}(s)\neq0$ for $p \leq y_1$. 
Using Lemma \ref{lem:3.4}, we also obtain $F_{\sigma,p}(s)\neq0$ for $p>y_1$ if $\delta$ is small enough. 
Therefore one can define $f_{\sigma,p}(s)=\log{F}_{\sigma,p}(s)$ as a holomorphic function on $|s-\kappa|\leq \kappa/2$ for any prime number $p$, where the branch is chosen so that $f_{\sigma,p}(u)\in \mathbb{R}$ if $u>0$. 
Then formula \eqref{eq:02262021} yields
\begin{gather}\label{eq:02282344}
f_\sigma(s)
=\sum_{p}
f_{\sigma,p}(s). 
\end{gather}
Furthermore, we immediately deduce the following results from the above lemmas. 

\begin{lemma}\label{lem:3.6}
Let $\kappa>0$ be a large real number. 
Then the following asymptotic formulas hold uniformly in the disk $|s-\kappa|\leq \kappa/2$. 
\begin{itemize}
\item[$(\mathrm{i})$] 
For any $p \leq y_1$, we have 
\begin{gather*}
f_{\sigma,p}(s)
=2s \lambda_{p,\sigma}(0)
-\frac{3}{2}\log(s|\lambda''_{p,\sigma}(0)|)
-\frac{1}{2}\log{4\pi}
+O\left(\frac{1}{\sqrt{\kappa}}+p^{-1} \right), 
\end{gather*}
where $\lambda_{p,\sigma}(\theta)$ is defined as \eqref{eq:02262144}. 
\item[$(\mathrm{ii})$] 
For any $y_1<p<y_2$, we have 
\begin{gather*}
f_{\sigma,p}(s)
=g(sp^{-\sigma})
+O\left(\kappa p^{-2\sigma}+p^{-1}\right), 
\end{gather*}
where $g(z)=\log{G}(z)$ as in Section \ref{sec:2}.  
\item[$(\mathrm{iii})$] 
For any $p \geq y_2$, we have $f_{\sigma,p}(s)=O\left(\kappa^2 p^{-2\sigma}\right)$.
\end{itemize}
\end{lemma}

\begin{proof}[Proof of Proposition \ref{prop:3.1}]
By \eqref{eq:02282344}, we obtain
\begin{gather}\label{eq:02282345}
f_\sigma^{(j)}(\kappa)
=\Bigg(\sum_{p \leq y_1} 
+\sum_{y_1<p<y_2} 
+\sum_{p \geq y_2} \Bigg)
f_{\sigma,p}^{(j)}(\kappa)
\end{gather}
for all $j \geq0$. 
Let $p \leq y_1$. 
We deduce from Lemma \ref{lem:3.6} $(\mathrm{i})$ the formula
\begin{gather*}
f_{\sigma,p}(\kappa)
=2\kappa \lambda_{p,\sigma}(0)
+O\left(\log(\kappa p^{-\sigma})\right)
\end{gather*}
by recalling $\lambda''_{p,\sigma}(0) \asymp p^{-\sigma}$. 
Then, we apply Cauchy's integral formula in a way similar to the proof of Lemma \ref{lem:2.3}. 
We obtain
\begin{gather*}
f'_{\sigma,p}(\kappa)
=2\lambda_{p,\sigma}(0)
+O\left(\kappa^{-1}\right)
\quad\text{and}\quad
f_{\sigma,p}^{(j)}(\kappa) 
\ll 2^jj! \kappa^{-j}
\end{gather*}
for all $j \geq2$. 
Therefore, the upper bounds $f_{\sigma,p}^{(j)}(\kappa) \ll 2^jj! \kappa^{-j+1}p^{-\sigma}$ hold for all $j \geq0$ since $\kappa p^{-\sigma} \gg \sqrt{\kappa}$ for $p \leq y_1$. 
Recalling that $y_1 \asymp \kappa^{1/2\sigma}$, we estimate the first sum of \eqref{eq:02282345} as
\begin{gather*}
\sum_{p \leq y_1} f_{\sigma,p}^{(j)}(\kappa)
\ll 2^jj! \kappa^{-j+1} \frac{y_1^{1-\sigma}}{\log{y_1}}
\ll 2^jj! \frac{\kappa^{\frac{1}{2\sigma}+\frac{1}{2}-j}}{\log{\kappa}}. 
\end{gather*}
Then, we consider the third sum of \eqref{eq:02282345}. 
Let $p \geq y_2$. 
We use Lemma \ref{lem:3.6} $(\mathrm{iii})$ to derive the bounds $f_{\sigma,p}^{(j)}(\kappa) \ll 2^jj! \kappa^{-j+2}p^{-2\sigma}$ for all $j \geq0$. 
Thus we obtain
\begin{gather*}
\sum_{p \geq y_2} f_{\sigma,p}^{(j)}(\kappa)
\ll 2^jj! \kappa^{-j+2} \frac{y_2^{1-2\sigma}}{\log{y_2}}
\ll 2^jj! \frac{\kappa^{\frac{1}{\sigma}-j}}{(\log{\kappa})^{N+1}}
\end{gather*}
from the choice of $y_2$. 
The main term comes from the second sum of \eqref{eq:02282345}. 
Using Lemma \ref{lem:3.6} $(\mathrm{ii})$, we have 
\begin{gather*}
f_{\sigma,p}^{(j)}(\kappa)
=p^{-j\sigma} g^{(j)}(\kappa p^{-\sigma})
+O\left(2^j j! \kappa^{1-j} p^{-2\sigma}
+2^j j! \kappa^{-j} p^{-1}\right)
\end{gather*}
for $y_1<p<y_2$. 
Note that 
\begin{gather*}
\sum_{y_1<p<y_2} \left(2^j j! \kappa^{1-j} p^{-2\sigma} +2^j j! \kappa^{-j} p^{-1}\right)
\ll 2^jj! \frac{\kappa^{\frac{1}{2\sigma}-j}}{\log{\kappa}} +2^jj! \kappa^{-j} \log_2\kappa. 
\end{gather*}
Put $u(y)=\kappa y^{-\sigma}$. 
Then we obtain 
\begin{gather*}
\sum_{y_1<p<y_2} f_{\sigma,p}^{(j)}(\kappa)
=\kappa^{-j} \sum_{y_1<p<y_2} u(p)^j g^{(j)}(u(p))
+O\left(2^jj! \frac{\kappa^{\frac{1}{2\sigma}-j}}{\log{\kappa}}
+2^jj! \kappa^{-j} \log_2\kappa \right). 
\end{gather*}
From the above, we deduce 
\begin{gather}\label{eq:03011723}
f_\sigma^{(j)}(\kappa)
=\kappa^{-j} \sum_{y_1<p<y_2} u(p)^j g^{(j)}(u(p))
+O\left(2^jj! \frac{\kappa^{\frac{1}{\sigma}-j}}{(\log{\kappa})^{N+1}} \right)
\end{gather}
for all $j \geq0$ with implied constants depending on $\sigma$ and $N$. 
We approximate the sum in \eqref{eq:03011723} by using the prime number theorem in the form 
\begin{gather*}
\pi(y)
=\int_{2}^{y}\frac{dy}{\log{y}}
+O\left(y e^{-8\sqrt{\log{y}}}\right), 
\end{gather*}
where $\pi(y)$ counts as usual the prime numbers not exceeding $y$. 
We obtain
\begin{gather}\label{eq:03021955}
\sum_{y_1<p<y_2} u(p)^j g^{(j)}(u(p))
=\int_{y_1}^{y_2} u(y)^j g^{(j)}(u(y))\frac{dy}{\log{y}}
+E
\end{gather}
by the partial summation, where 
\begin{align*}
E
&\ll u(y_1)^j g^{(j)}(u(y_1)) y_1 e^{-8\sqrt{\log{y_1}}}
+u(y_2)^j g^{(j)}(u(y_2)) y_2 e^{-8\sqrt{\log{y_2}}} \\
&\qquad +e^{-8\sqrt{\log{y_1}}} \int_{y_1}^{y_2} u(y)^j g^{(j)}(u(y)) \,dy. 
\end{align*}
We put $y_3=\kappa^{1/\sigma}$. 
By the choices of these parameters, we have $0<u(y_2)<1$, $u(y_3)=1$, and $u(y_1)>1$. 
Hence Lemma \ref{lem:2.3} yields 
\begin{align*}
E 
&\ll 2^jj! u(y_1) y_1 e^{-8\sqrt{\log{y_1}}}
+j! u(y_2)^2 y_2 e^{-8\sqrt{\log{y_2}}} \\
&\qquad+2^jj! e^{-8\sqrt{\log{y_1}}}\int_{y_1}^{y_3} u(y) \,dy
+2^jj! e^{-8\sqrt{\log{y_1}}}\int_{y_3}^{y_2} u(y)^2 \,dy \nonumber\\
&\ll2^jj! \frac{\kappa^{\frac{1}{\sigma}}}{(\log{\kappa})^{N+1}}.  \nonumber
\end{align*}
The integral in \eqref{eq:03021955} is calculated as follows. 
Changing the variables, we have
\begin{gather*}
\int_{y_1}^{y_2} u(y)^j g^{(j)}(u(y))\frac{dy}{\log{y}}
=\kappa^{\frac{1}{\sigma}}
\int_{u_2}^{u_1} \frac{g^{(j)}(u)}{u^{\frac{1}{\sigma}+1-j}} \frac{1}{\log(u/\kappa)}\,du, 
\end{gather*}
where $u_1=u(y_1)$ and $u_2=u(y_2)$. 
For $u_2 \leq u \leq u_1$, the asymptotic formula
\begin{gather*}
\frac{1}{\log(u/\kappa)}
=\frac{1}{\log{\kappa}}
\Bigg\{\sum_{n=0}^{N-1}\left(\frac{\log{u}}{\log{\kappa}}\right)^n 
+O\left(\left(\frac{|\log{u}|}{\log{\kappa}}\right)^N\right)\Bigg\}
\end{gather*}
is valid, which yields
\begin{align*}
\int_{y_1}^{y_2} u(y)^j g^{(j)}(u(y))\frac{dy}{\log{y}} 
&=\frac{\kappa^{\frac{1}{\sigma}}}{\log{\kappa}}
\Bigg\{\sum_{n=0}^{N-1}\frac{1}{(\log{\kappa})^n}
\int_{u_2}^{u_1} \frac{g^{(j)}(u)}{u^{\frac{1}{\sigma}+1-j}}(\log{u})^n \,du \\
&\qquad\quad+O\left(\frac{1}{(\log{\kappa})^N}
\int_{u_2}^{u_1} \frac{|g^{(j)}(u)|}{u^{\frac{1}{\sigma}+1-j}}|\log{u}|^N \,du\right) \Bigg\}. 
\end{align*}
Let $\mathfrak{g}_{n,j}(\sigma)$ denote the constants of \eqref{eq:02261606}. 
Then we have 
\begin{gather*}
\int_{u_2}^{u_1} \frac{g^{(j)}(u)}{u^{\frac{1}{\sigma}+1-j}}(\log{u})^n \,du
=\mathfrak{g}_{n,j}(\sigma)+E_{n,j} 
\end{gather*}
for all $n=0,\ldots,N-1$, where $E_{n,j}$ are evaluated as
\begin{gather*}
E_{n,j}
\ll 2^jj!\int_{0}^{u_2} u^{1-\frac{1}{\sigma}}|\log{u}|^n \,du
+2^jj!\int_{u_1}^{\infty} u^{-\frac{1}{\sigma}}|\log{u}|^n \,du
\ll \frac{2^jj!}{(\log{\kappa})^N}
\end{gather*}
by Lemma \ref{lem:2.3} together with $0<u_2<1$ and $u_1>1$. 
Furthermore, we obtain
\begin{gather*}
\int_{u_2}^{u_1} \frac{|g^{(j)}(u)|}{u^{\frac{1}{\sigma}+1-j}}|\log{u}|^N \,du
\leq \int_{0}^{\infty} \frac{|g^{(j)}(u)|}{u^{\frac{1}{\sigma}+1-j}}|\log{u}|^N \,du
\ll2^jj!. 
\end{gather*}
As a result, we find that the integral is estimated as 
\begin{gather*}
\int_{y_1}^{y_2} u(y)^j g^{(j)}(u(y))\frac{dy}{\log{y}}
=\frac{\kappa^{\frac{1}{\sigma}}}{\log{\kappa}}
\Bigg\{\sum_{n=0}^{N-1}\frac{\mathfrak{g}_{n,j}(\sigma)}{(\log{\kappa})^n}
+O\left(\frac{2^jj!}{(\log{\kappa})^N}\right)\Bigg\}. 
\end{gather*}
Combined with \eqref{eq:03011723} and \eqref{eq:03021955}, it deduces the asymptotic formula
\begin{align*}
f_\sigma^{(j)}(\kappa)
&=\kappa^{-j}\int_{y_1}^{y_2} u(y)^j g^{(j)}(u(y))\frac{dy}{\log{y}}
+O\left(2^jj! \frac{\kappa^{\frac{1}{\sigma}-j}}{(\log{\kappa})^{N+1}} \right)\\
&=\frac{\kappa^{\frac{1}{\sigma}-j}}{\log{\kappa}}
\left\{\sum_{n=0}^{N-1} \frac{\mathfrak{g}_{n,j}(\sigma)}{(\log{\kappa})^n}
+O\left(\frac{2^jj!}{(\log{\kappa})^N} \right)\right\}, 
\end{align*}
which completes the proof of Proposition \ref{prop:3.1}. 
\end{proof}

\begin{proof}[Proof of Proposition \ref{prop:3.2}]
If $j \geq2$, then the desired formula of $f_1^{(j)}(\kappa)$ can be shown by an argument similar to the proof of Proposition \ref{prop:3.1}. 
We hereby present the proof when $j=0$. 
By Lemma \ref{lem:3.6} $(\mathrm{i})$, we have 
\begin{gather*}
f_{1,p}(\kappa)
=-2\kappa \log(1-p^{-1})
+O\left(\log(\kappa p^{-1})\right)
\end{gather*}
for any $p \leq y_1$. 
Recalling the definition of $g_*(u)$, we deduce from Lemma \ref{lem:3.6} $(\mathrm{ii})$ the formulas
\begin{gather*}
f_{1,p}(\kappa)
=
\begin{cases}
g_*(\kappa p^{-1})+2\kappa p^{-1}+O\left(\kappa p^{-2}+p^{-1}\right)
&\text{if $y_1<p<\kappa$}, 
\\
g_*(\kappa p^{-1})+O\left(\kappa p^{-2}+p^{-1}\right)
&\text{if $\kappa \leq p<y_2$}.
\end{cases}
\end{gather*}
Therefore the sum of terms for $p<y_2$ is calculated as
\begin{gather*}
\sum_{p<y_2}f_{1,p}(\kappa)
=-2\kappa \sum_{p<\kappa} \log(1-p^{-1})
+\sum_{y_1<p<y_2} g_*(\kappa p^{-1})
+E,
\end{gather*}
where the error term $E$ is evaluated as
\begin{align*}
E
&=2\kappa \sum_{y_1<p<\kappa} \left\{\log(1-p^{-1})+p^{-1}\right\} \\
&\qquad\quad+O\left(\sum_{p<y_1}\log(\kappa p^{-1})
+\sum_{y_1<p<y_2} \left(\kappa p^{-2}+p^{-1}\right)\right).
\end{align*}
Furthermore, we use the estimate $\log(1-x)=-x+O(x^2)$ with $|x|<1$ to obtain 
\begin{align*}
E
&\ll \sum_{y_1<p<\kappa} \kappa p^{-2}
+\sum_{p<y_1}\log(\kappa p^{-1})
+\sum_{y_1<p<y_2} \left(\kappa p^{-2}+p^{-1}\right) \\
&\ll \frac{\kappa}{y_1 \log{y_1}}
+\frac{(\log{\kappa})y_1}{\log{y_1}}
+\log_2(y_2)
\ll \sqrt{\kappa}. 
\end{align*}
Applying the asymptotic formula
\begin{gather*}
-\sum_{p<y} \log(1-p^{-1})
=\log_2{y}+\gamma+O\left(e^{-2\sqrt{\log{y}}}\right), 
\end{gather*}
we derive 
\begin{gather}\label{eq:03030103}
\sum_{p<y_2}f_{1,p}(\kappa)
=2\kappa(\log_2{\kappa}+\gamma)
+\sum_{y_1<p<y_2} g_*(\kappa p^{-1})
+O\left(\kappa e^{-2\sqrt{\log{\kappa}}}\right). 
\end{gather}
Then we estimate the contribution of terms for $p \geq y_2$. 
By Lemma \ref{lem:3.6} $(\mathrm{iii})$, we obtain 
\begin{gather}\label{eq:03030104}
\sum_{p \geq y_2} f_{1,p}(\kappa)
\ll \frac{\kappa^2}{y_2 \log{y_2}}
\ll \frac{\kappa}{(\log{\kappa})^{N+1}},  
\end{gather}
where the implied constant depends on $N$. 
Combining \eqref{eq:03030103} and \eqref{eq:03030104}, we obtain 
\begin{gather}\label{eq:03030114}
f_1(\kappa)
=2\kappa(\log_2{\kappa}+\gamma)
+\sum_{y_1<p<y_2} g_*(\kappa p^{-1})
+O\left(\frac{\kappa}{(\log{\kappa})^{N+1}}\right). 
\end{gather}
The work of estimating the sum in \eqref{eq:03030114} remains, but one can show 
\begin{gather*}
\sum_{y_1<p<y_2} g_*(\kappa p^{-1})
=\frac{\kappa}{\log{\kappa}}
\Bigg\{\sum_{n=0}^{N-1}\frac{\mathfrak{g}_{n,0}}{(\log{\kappa})^n}
+O\left(\frac{1}{(\log{\kappa})^N}\right)\Bigg\}
\end{gather*}
along the same line as the proof of Proposition \ref{prop:3.1}. 
Hence the desired asymptotic formula of $f_1(\kappa)$ follows. 
In addition, the proof for $f'_1(\kappa)$ is given in a similar way, and we omit the proof. 
\end{proof}

\section{Probability density functions}\label{sec:4}

\subsection{Approximation of the saddle-point}\label{sec:4.1}
As seen in \cite{Lamzouri2010, Lamzouri2011b, LiuRoyerWu2008}, the saddle-point method is useful to show asymptotic formulas such as \eqref{eq:02151442} and \eqref{eq:02231541}. 
Here, the saddle-point stands for the solution $\kappa=\kappa(\sigma,\tau)$ to the equation
\begin{gather}\label{eq:03031628}
f'_\sigma(\kappa)
=\tau, 
\end{gather}
where $f_\sigma(\kappa)$ is the cumulant-generating function defined by \eqref{eq:02261551}. 
We approximate the saddle-point by extending the method of Liu--Royer--Wu \cite{LiuRoyerWu2008}. 
First, we prove the existence and uniqueness of the solution of \eqref{eq:03031628}. 

\begin{lemma}\label{lem:4.1}
Let $\sigma>1/2$ and $\tau>0$. 
Then there exists a unique real number $\kappa=\kappa(\sigma,\tau)>0$ for which \eqref{eq:03031628} is satisfied. 
We have $\kappa(\sigma,\tau)\to\infty$ as $\tau \to\infty$. 
\end{lemma}

\begin{proof}
Since $f_\sigma(\kappa)=\log{F}_\sigma(\kappa)$, we calculate its second derivative as
\begin{align*}
f''_\sigma(\kappa)
&=\frac{F''_\sigma(\kappa)F_\sigma(\kappa)-F'_\sigma(\kappa)^2}{F_\sigma(\kappa)^2} \\
&=\frac{1}{F_\sigma(\kappa)}
\int_{\mathbb{R}} (x-f'_\sigma(\kappa))^2 e^{\kappa x} \mathcal{M}_\sigma(x) \,|dx|. 
\end{align*}
This implies $f''_\sigma(\kappa)>0$, and thus $f'_\sigma(\kappa)$ is strictly increasing for $\kappa>0$. 
Furthermore, 
\begin{gather*}
\mathbb{E}\left[\log{L}(\sigma,\Theta)\right]
=\sum_{p} \sum_{m=1}^{\infty} \frac{1}{m} \mathbb{E}\left[\cos(m\Theta_p)\right] p^{-m \sigma}
=0
\end{gather*}
since $\mathbb{E}\left[\cos(m\Theta_p)\right]=0$ for all $m \geq1$. 
Hence the value $f'_\sigma(0)$ is calculated as 
\begin{gather*}
f'_\sigma(0)
=\frac{F'_\sigma(0)}{F_\sigma(0)}
=\mathbb{E}\left[\log{L}(\sigma,\Theta)\right]
=0. 
\end{gather*}
Therefore we obtain the conclusion.  
\end{proof}

Then, we apply Propositions \ref{prop:3.1} and \ref{prop:3.2} to derive the following results on the saddle-point $\kappa=\kappa(\sigma,\tau)$ for $1/2<\sigma \leq1$. 

\begin{proposition}\label{prop:4.2}
Let $N \in \mathbb{Z}_{\geq1}$. 
Denote by $\kappa=\kappa(\sigma,\tau)$ the solution of \eqref{eq:03031628} with $1/2<\sigma<1$. 
For $n=0,\ldots,N-1$, there exist polynomials $B_{n,\sigma}(x)$ of degree at most $n$ with $B_{0,\sigma}(x)=1$ such that the formula 
\begin{gather*}
\kappa
=B(\sigma) (\tau \log{\tau})^{\frac{\sigma}{1-\sigma}} 
\left\{\sum_{n=0}^{N-1} \frac{B_{n,\sigma}(\log_2{\tau})}{(\log{\tau})^n}
+O\left(\left(\frac{\log_2{\tau}}{\log{\tau}}\right)^N\right)\right\}
\end{gather*}
holds if $\tau>0$ is large enough. 
Here, the implied constant depends on $\sigma$ and $N$, and $B(\sigma)$ is the positive constant determined as
\begin{gather}\label{eq:03201406}
B(\sigma)
=\left(\frac{1-\sigma}{\sigma} \mathfrak{g}_{0,1}(\sigma)\right)^{-\frac{\sigma}{1-\sigma}}.  
\end{gather}
Furthermore, the polynomial $B_{1,\sigma}(x)$ is obtained as
\begin{gather*}
B_{1,\sigma}(x)
=\frac{\sigma}{1-\sigma}x+\log{B(\sigma)}-\frac{\mathfrak{g}_{1,1}(\sigma)}{\mathfrak{g}_{0,1}(\sigma)},  
\end{gather*}
where $\mathfrak{g}_{0,1}(\sigma)$ and $\mathfrak{g}_{1,1}(\sigma)$ are the constants defined by \eqref{eq:02261606}. 
\end{proposition}

\begin{proof}
First, we apply Proposition \ref{prop:3.1} with $N=1$. 
It yields the formula
\begin{gather}\label{eq:03140110}
\tau
=\mathfrak{g}_{0,1}(\sigma) \frac{\kappa^{\frac{1-\sigma}{\sigma}}}{\log{\kappa}}
\left(1+O\left(\frac{1}{\log{\kappa}}\right)\right).  
\end{gather}
Furthermore, the logarithm is estimated as
\begin{gather}\label{eq:03140111}
\log{\tau}
=\frac{1-\sigma}{\sigma}(\log{\kappa})
\left(1+O\left(\frac{\log_2{\kappa}}{\log{\kappa}}\right)\right), 
\end{gather}
and therefore $\log{\tau} \asymp \log{\kappa}$ follows. 
Put $\kappa=B(\sigma) (\tau \log{\tau})^{\frac{\sigma}{1-\sigma}} \left(1+h_\sigma(\tau)\right)$. 
Then we deduce from \eqref{eq:03140110} and \eqref{eq:03140111} that $h_\sigma(\tau)$ satisfies
\begin{gather}\label{eq:05062328}
h_\sigma(\tau)
\ll \frac{\log_2{\kappa}}{\log{\kappa}} 
\ll \frac{\log_2{\tau}}{\log{\tau}},  
\end{gather}
which derives the result when $N=1$. 
To consider the case $N=2$, we calculate the terms $\kappa^{\frac{1-\sigma}{\sigma}}$ and $\log{\kappa}$ as
\begin{align*}
\kappa^{\frac{1-\sigma}{\sigma}}
&=B(\sigma)^{\frac{1-\sigma}{\sigma}} (\tau \log{\tau}) 
\left\{1+\frac{1-\sigma}{\sigma} h_\sigma(\tau)
+O\left(\left(\frac{\log_2{\tau}}{\log{\tau}}\right)^2\right)\right\}, \\
\log{\kappa}
&=\frac{\sigma}{1-\sigma} (\log{\tau})
\left\{1+\frac{\log_2{\tau}+\frac{1-\sigma}{\sigma}\log{B(\sigma)}}{\log{\tau}} 
+O\left(\left(\frac{\log_2{\tau}}{\log{\tau}}\right)^2\right)\right\} 
\end{align*}
by using $\kappa=B(\sigma) (\tau \log{\tau})^{\frac{\sigma}{1-\sigma}} \left(1+h_\sigma(\tau)\right)$ and \eqref{eq:05062328}. 
We insert them to the formula
\begin{gather*}
\tau
=\mathfrak{g}_{0,1}(\sigma) \frac{\kappa^{\frac{1-\sigma}{\sigma}}}{\log{\kappa}}
\Bigg\{1+\frac{\mathfrak{g}_{1,1}(\sigma)}{\mathfrak{g}_{0,1}(\sigma)} \frac{1}{\log{\kappa}}
+O\left(\frac{1}{(\log{\kappa})^2} \right)\Bigg\}
\end{gather*}
which is a consequence of Proposition \ref{prop:3.1} with $N=2$. 
Then the identity
\begin{align*}
\tau
&=\tau
\Bigg\{1+\frac{1-\sigma}{\sigma} h_\sigma(\tau) 
-\frac{\log_2{\tau}+\frac{1-\sigma}{\sigma}\log{B(\sigma)}}{\log{\tau}} \\
&\qquad\qquad\qquad
+\frac{\mathfrak{g}_{1,1}(\sigma)}{\mathfrak{g}_{0,1}(\sigma)} \frac{1-\sigma}{\sigma} (\log{\tau})^{-1}
+O\left(\left(\frac{\log_2{\tau}}{\log{\tau}}\right)^2\right)\Bigg\}
\end{align*}
follows. 
Therefore, we see that $h_\sigma(\tau)$ satisfies
\begin{align*}
h_\sigma(\tau)
&=\frac{1}{\log{\tau}} \left(\frac{\sigma}{1-\sigma}\log_2{\tau}+\log{B(\sigma)}
-\frac{\mathfrak{g}_{1,1}(\sigma)}{\mathfrak{g}_{0,1}(\sigma)}\right)
+O\left(\left(\frac{\log_2{\tau}}{\log{\tau}}\right)^2\right) \\
&=\frac{B_{1,\sigma}(\log_2{\tau})}{\log{\tau}} +O\left(\left(\frac{\log_2{\tau}}{\log{\tau}}\right)^2\right). 
\end{align*}
It derives the desired result when $N=2$. 
For $m \geq2$, we assume that it is valid further when $N=1,\ldots,m$. 
Note that the asymptotic formula
\begin{gather}\label{eq:03171801}
\tau
=\mathfrak{g}_{0,1}(\sigma) \frac{\kappa^{\frac{1-\sigma}{\sigma}}}{\log{\kappa}}
\Bigg\{1+\sum_{j=1}^{m} \frac{\mathfrak{g}_{j,1}(\sigma)}{\mathfrak{g}_{0,1}(\sigma)} \frac{1}{(\log{\kappa})^j}
+O\left(\frac{1}{(\log{\kappa})^{m+1}} \right)\Bigg\} 
\end{gather}
follows from Proposition \ref{prop:3.1} with $N=m+1$. 
If we put
\begin{gather*}
\kappa
=B(\sigma)(\tau \log{\tau})^{\frac{\sigma}{1-\sigma}} 
\left\{\sum_{n=0}^{m-1} \frac{B_{n,\sigma}(\log_2{\tau})}{(\log{\tau})^n}
+h_{m,\sigma}(\tau)\right\}, 
\end{gather*}
then the inductive assumption gives the upper bound
\begin{gather}\label{eq:03121713}
h_{m,\sigma}(\tau)
\ll \left(\frac{\log_2{\tau}}{\log{\tau}}\right)^m. 
\end{gather}
Hence the terms $\kappa^{\frac{1-\sigma}{\sigma}}$ and $\log{\kappa}$ are calculated as
\begin{align*}
\kappa^{\frac{1-\sigma}{\sigma}}
&=B(\sigma)^{\frac{1-\sigma}{\sigma}} (\tau \log{\tau}) 
\Bigg\{1+\sum_{n=1}^{m} \frac{C_{n,\sigma}(\log_2{\tau})}{(\log{\tau})^n} \\
&\qquad\qquad\qquad\qquad\qquad
+\frac{1-\sigma}{\sigma} h_{m,\sigma}(\tau)
+O\left(\left(\frac{\log_2{\tau}}{\log{\tau}}\right)^{m+1}\right)\Bigg\}, \\
\log{\kappa}
&=\frac{\sigma}{1-\sigma} (\log{\tau})
\left\{1+\sum_{n=1}^{m} \frac{C_{n,\sigma}^*(\log_2{\tau})}{(\log{\tau})^n} 
+O\left(\left(\frac{\log_2{\tau}}{\log{\tau}}\right)^{m+1}\right)\right\},
\end{align*}
where $C_{n,\sigma}(x)$ and $C_{n,\sigma}^*(x)$ are polynomials of degree at most $n$. 
Then formula \eqref{eq:03171801} deduces the identity
\begin{gather*}
\tau
=\tau \left\{1+\sum_{n=1}^{m} \frac{D_{n,\sigma}(\log_2{\tau})}{(\log{\tau})^n} 
+\frac{1-\sigma}{\sigma} h_{m,\sigma}(\tau)
+O\left(\left(\frac{\log_2{\tau}}{\log{\tau}}\right)^{m+1}\right)\right\}
\end{gather*}
with some polynomials $D_{n,\sigma}(x)$ of degree at most $n$. 
Therefore we obtain
\begin{gather*}
h_{m,\sigma}(\tau)
=-\frac{\sigma}{1-\sigma}\sum_{n=1}^{m} \frac{D_{n,\sigma}(\log_2{\tau})}{(\log{\tau})^n}
+O\left(\left(\frac{\log_2{\tau}}{\log{\tau}}\right)^{m+1}\right). 
\end{gather*}
By inductive assumption \eqref{eq:03121713}, we see that $D_{1,\sigma}(x)=\cdots=D_{m-1,\sigma}(x)=0$. 
Hence the formula
\begin{gather*}
h_{m,\sigma}(\tau)
=-\frac{\sigma}{1-\sigma} \frac{D_{m,\sigma}(\log_2{\tau})}{(\log{\tau})^m}
+O\left(\left(\frac{\log_2{\tau}}{\log{\tau}}\right)^{m+1}\right) 
\end{gather*}
follows, which asserts that the desired result is valid when $N=m+1$. 
\end{proof}

\begin{proposition}\label{prop:4.3}
Let $N \in \mathbb{Z}_{\geq1}$ and put $\tau=2\log{t}+2\gamma$.  
Denote by $\kappa=\kappa(1,\tau)$ the solution of \eqref{eq:03031628} with $\sigma=1$. 
For $n=0,\ldots,N-1$, there exist real numbers $b_n$ with $b_0=1$ such that the formula
\begin{gather*}
\kappa
=\exp\left(t-\frac{1}{2}\mathfrak{g}_{0,1}\right) 
\left\{ \sum_{n=0}^{N-1} \frac{b_n}{t^n}
+O\left(\frac{1}{t^N}\right) \right\}
\end{gather*}
holds if $t>0$ is large enough. 
Here, the implied constant depends on $N$. 
Furthermore, the real number $b_1$ is obtained as
\begin{gather*}
b_1
=-\frac{1}{8}\mathfrak{g}_{0,1}^2-\frac{1}{2}\mathfrak{g}_{1,1}, 
\end{gather*}
where $\mathfrak{g}_{0,1}$ and $\mathfrak{g}_{1,1}$ are the constants defined by \eqref{eq:02261607}. 
\end{proposition}

\begin{proof}
By Proposition \ref{prop:3.2} with $N=1$, we have 
\begin{gather*}
2\log{t}
=2\log_2{\kappa}
+\frac{\mathfrak{g}_{0,1}}{\log{\kappa}}
+O\left(\frac{1}{(\log{\kappa})^2}\right) 
\end{gather*}
since we put $\tau=2\log{t}+2\gamma$. 
Therefore the asymptotic formula
\begin{gather*}
t
=(\log{\kappa}) \exp\left(\frac{\mathfrak{g}_{0,1}}{2\log{\kappa}}
+O\left(\frac{1}{(\log{\kappa})^2}\right)\right)
=\log{\kappa}+\frac{1}{2}\mathfrak{g}_{0,1}
+O\left(\frac{1}{t}\right)
\end{gather*}
follows by noting that $t \asymp \log{\kappa}$ holds. 
If we use Proposition \ref{prop:3.2} with $N \geq2$, then one can prove more generally
\begin{gather}\label{eq:03201702}
t
=\log{\kappa}
+\frac{1}{2}\mathfrak{g}_{0,1}
+\sum_{n=1}^{N-1} \frac{\beta_n}{t^n}
+O\left(\frac{1}{t^N}\right)
\end{gather}
by induction on $N$, where $\beta_n$ are real numbers such that $\beta_1=\frac{1}{8}\mathfrak{g}_{0,1}^2+\frac{1}{2}\mathfrak{g}_{1,1}$. 
See also the proof of \cite[Lemma 8.1]{LiuRoyerWu2008} for an analogous argument. 
From the above, we obtain the asymptotic formula
\begin{align*}
\kappa
&=\exp\left(t-\frac{1}{2}\mathfrak{g}_{0,1}
-\sum_{n=1}^{N-1} \frac{\beta_n}{t^n}
+O\left(\frac{1}{t^N}\right)\right) \\
&=\exp\left(t-\frac{1}{2}\mathfrak{g}_{0,1}\right)
\left\{ \sum_{n=0}^{N-1} \frac{b_n}{t^n}
+O\left(\frac{1}{t^N}\right) \right\}
\end{align*}
as desired, where $b_n$ are real numbers such that $b_0=1$ and $b_1=-\beta_1$. 
\end{proof}

\subsection{Proof of Theorem \ref{thm:1.1}}\label{sec:4.2}
Let $\sigma>1/2$ and $\tau>0$. 
Denote by $\kappa=\kappa(\sigma,\tau)$ the solution of \eqref{eq:03031628}. 
Using the density function $\mathcal{M}_\sigma$ of \eqref{eq:05111549}, we define
\begin{gather}\label{eq:03211608}
\mathcal{N}_\sigma(x;\tau)
=\frac{e^{\kappa(x+\tau)}}{F_\sigma(\kappa)}
\mathcal{M}_\sigma(x+\tau)
\end{gather}
for $x \in \mathbb{R}$, where $F_\sigma(\kappa)$ is the moment-generating function of \eqref{eq:02261551}. 
To begin with, we show the following lemmas on the function $\mathcal{N}_\sigma(x;\tau)$. 

\begin{lemma}\label{lem:4.4}
Let $\sigma>1/2$ and $\tau>0$. 
Then $\mathcal{N}_\sigma(u;\tau)$ is a non-negative continuous function satisfying the equalities
\begin{gather*}
\int_{\mathbb{R}} \mathcal{N}_\sigma(x;\tau) \,|dx|
=1
\quad\text{and}\quad
\int_{\mathbb{R}} \mathcal{N}_\sigma(x;\tau)x \,|dx|
=0. 
\end{gather*}
\end{lemma}

\begin{proof}
By the definition of $\mathcal{N}_\sigma$, the Fourier transform is represented as
\begin{gather}\label{eq:03191630}
\widetilde{\mathcal{N}}_\sigma(v;\tau)
:=\int_{\mathbb{R}} \mathcal{N}_\sigma(x;\tau) e^{ixv} \,|dx|
=e^{-i \tau v} \frac{F_\sigma(\kappa+iv)}{F_\sigma(\kappa)}.
\end{gather}
Thus we have $\widetilde{\mathcal{N}}_\sigma(0;\tau)=1$, and furthermore, 
\begin{gather*}
\frac{d}{dv} \widetilde{\mathcal{N}}_\sigma(v;\tau) \bigg|_{v=0}
=-i \tau +i f'_\sigma(\kappa)
=0
\end{gather*}
due to \eqref{eq:03031628}. 
Hence the result follows since we have the identities
\begin{gather*}
\frac{d^k}{dv^k} \widetilde{\mathcal{N}}_\sigma(v;\tau) \bigg|_{v=0}
=i^k \int_{\mathbb{R}} \mathcal{N}_\sigma(x;\tau)x^k \,|dx|
\end{gather*}
for all $k \geq0$. 
\end{proof}

\begin{lemma}\label{lem:4.5}
Let $1/2<\sigma \leq1$ and $\tau>0$ be a large real number. 
Then there exist positive constants $c_1(\sigma)$ and $c_2(\sigma)$ such that we have 
\begin{gather*}
\big|\widetilde{\mathcal{N}}_\sigma(v;\tau)\big|
\leq \exp\left(-c_2(\sigma)v^2 \frac{\kappa^{\frac{1}{\sigma}-2}}{\log{\kappa}} \right)
\end{gather*}
if $|v| \leq c_1(\sigma) \kappa$ is satisfied. 
\end{lemma}

\begin{proof}
By formula \eqref{eq:03191630}, it is sufficient to evaluate $|F_\sigma(s)|/F_\sigma(\kappa)$ with $s=\kappa+iv$. 
Recall that the function $F_\sigma(s)$ satisfies \eqref{eq:02262021}. 
Then we obtain
\begin{gather}\label{eq:03301309}
\frac{|F_\sigma(s)|}{F_\sigma(\kappa)}
\leq \prod_{Q_1<p<Q_2} \frac{|F_{\sigma,p}(s)|}{F_{\sigma,p}(\kappa)}
\end{gather}
for $Q_1,Q_2>0$ since the inequality $|F_{\sigma,p}(s)| \leq F_{\sigma,p}(\kappa)$ holds for every $p$. 
We deduce from Lemmas \ref{lem:2.1} and \ref{lem:3.4} that 
\begin{gather}\label{eq:03301325}
F_{\sigma,p}(s)
=\frac{1}{\sqrt{4\pi}} \frac{\exp(2sp^{-\sigma})}{(sp^{-\sigma})^{3/2}}
\left(1+h_{\sigma,p}(s)\right)
\end{gather}
for any $p>y_1$ in the disk $|s-\kappa|\leq \kappa/2$, where $h_{\sigma,p}(s)$ is a holomorphic function such that $h_{\sigma,p}(s)\ll \kappa^{-1}p^{\sigma}+\kappa p^{-2\sigma}+p^{-1}$. 
By Cauchy's integral formula, we have 
\begin{gather}\label{eq:03301331}
h_{\sigma,p}^{(j)}(\kappa)
\ll 2^jj!\kappa^{-j}
\left(\kappa^{-1}p^{\sigma}+\kappa p^{-2\sigma}+p^{-1}\right) 
\end{gather}
for all $j \geq0$. 
Then, we choose the parameters $Q_1, Q_2>0$ as 
\begin{gather}\label{eq:03301401}
Q_1
=\left(\frac{\kappa}{\epsilon_1} \right)^{\frac{1}{2\sigma}}
\quad\text{and}\quad
Q_2
=\left(\epsilon_2 \kappa \right)^{\frac{1}{\sigma}}
\end{gather}
with small positive constants $\epsilon_j=\epsilon_j(\sigma)$. 
For $Q_1<p<Q_2$, formula \eqref{eq:03301325} yields 
\begin{align*}
\log\frac{|F_{\sigma,p}(s)|}{F_{\sigma,p}(\kappa)}
&=-\frac{3}{2}\RE\log\left(1+\frac{iv}{\kappa}\right) 
+\RE\log\left(1+\frac{h_{\sigma,p}(s)-h_{\sigma,p}(\kappa)}{1+h_{\sigma,p}(\kappa)}\right) \\
&=-\frac{3}{4}\frac{v^2}{\kappa^2}+\frac{\RE{(h_{\sigma,p}(s)-h_{\sigma,p}(\kappa))}}{1+h_{\sigma,p}(\kappa)}
+O\left(\frac{v^3}{\kappa^3}+|h_{\sigma,p}(s)-h_{\sigma,p}(\kappa)|^2\right)
\end{align*} 
if $|v| \leq c_1 \kappa$ is satisfied with a small positive constant $c_1=c_1(\sigma)$. 
Remark that $h'_{\sigma,p}(\kappa)$ is real by definition. 
By \eqref{eq:03301331}, we have 
\begin{align*}
\RE{(h_{\sigma,p}(s)-h_{\sigma,p}(\kappa))}
&\ll \sum_{j=2}^{\infty} \frac{|h_{\sigma,p}^{(j)}(\kappa)|}{j!} |v|^j \\
&\ll \frac{v^2}{\kappa^2} \left(\kappa^{-1}p^{\sigma}+\kappa p^{-2\sigma}+p^{-1}\right), 
\end{align*}
and furthermore, 
\begin{gather*}
|h_{\sigma,p}(s)-h_{\sigma,p}(\kappa)|^2
\ll \frac{v^2}{\kappa^2} \left(\kappa^{-1}p^{\sigma}+\kappa p^{-2\sigma}+p^{-1}\right)
\end{gather*}
for $Q_1<p<Q_2$. 
Therefore we deduce
\begin{gather*}
\log\frac{|F_{\sigma,p}(s)|}{F_{\sigma,p}(\kappa)}
=\left(-\frac{3}{4}+O\left(\frac{|v|}{\kappa}+\kappa^{-1}p^{\sigma}+\kappa p^{-2\sigma}+p^{-1}\right)\right)
\frac{v^2}{\kappa^2} 
\leq -\frac{1}{2} \frac{v^2}{\kappa^2}
\end{gather*}
for $Q_1<p<Q_2$ if $c_1, \epsilon_1, \epsilon_2>0$ are small enough, and $\kappa=\kappa(\sigma,\tau)>0$ is large enough.  
By the prime number theorem, we obtain the inequality
\begin{gather*}
\sum_{Q_1<p<Q_2} \log\frac{|F_{\sigma,p}(s)|}{F_{\sigma,p}(\kappa)}
\leq -\frac{1}{4} \frac{v^2}{\kappa^2} \frac{Q_2}{\log{Q_2}}. 
\end{gather*}
Inserting \eqref{eq:03301401}, we conclude that
\begin{gather*}
\prod_{Q_1<p<Q_2} \frac{|F_{\sigma,p}(s)|}{F_{\sigma,p}(\kappa)}
\leq \exp\left(-c_2(\sigma)v^2 \frac{\kappa^{\frac{1}{\sigma}-2}}{\log{\kappa}} \right)
\end{gather*}
with some positive constant $c_2(\sigma)$, which completes the proof. 
\end{proof}

\begin{lemma}\label{lem:4.6}
Let $1/2<\sigma \leq1$ and $\tau>0$ be a large real number. 
For any $c>0$, there exists a positive constant $c_3(\sigma,c)$ such that we have 
\begin{gather*}
\big|\widetilde{\mathcal{N}}_\sigma(v;\tau)\big|
\leq \exp\left(-c_3(\sigma,c)\frac{|v|^{\frac{1}{\sigma}}}{\log{|v|}}\right)
\end{gather*}
if $|v|>c \kappa$ is satisfied.
\end{lemma}

\begin{proof}
Similarly to \eqref{eq:03301309}, we have the inequality
\begin{gather*}
\frac{|F_\sigma(s)|}{F_\sigma(\kappa)}
\leq \prod_{p>Q_3} \frac{|F_{\sigma,p}(s)|}{F_{\sigma,p}(\kappa)}
\end{gather*}
for $Q_3>0$. 
If the condition $|s|p^{-\sigma}<\delta$ is valid with a small positive constant $\delta=\delta(\sigma,c)$, then $F_{\sigma,p}(s)$ is calculated as
\begin{align*}
F_{\sigma,p}(s)
&=1+2s p^{-\sigma} \mathbb{E}\left[\cos \Theta_p\right]
+s^2 p^{-2\sigma} \mathbb{E}\left[(\cos \Theta_p)^2\right]
+O\left(|s|^3p^{-3\sigma}\right) \\
&=1+\frac{1}{4}s^2 p^{-2\sigma}
+O\left(|s|^2p^{-2\sigma-1}+|s|^3p^{-3\sigma}\right)
\end{align*}
by using the equalities 
\begin{gather*}
\mathbb{E}\left[\cos \Theta_p\right]
=0
\quad\text{and}\quad
\mathbb{E}\left[(\cos \Theta_p)^2\right]
=\frac{1}{4} \left(1+\frac{1}{p}\right). 
\end{gather*} 
It yields the asymptotic formula
\begin{gather*}
\log{F}_{\sigma,p}(s)
=\frac{1}{4}s^2 p^{-2\sigma}
+O\left(|s|^2p^{-2\sigma-1}+|s|^3p^{-3\sigma}\right) 
\end{gather*}
if $Q_3$ is large and $\delta$ is small. 
Then, we choose the parameter $Q_3>0$ as 
\begin{gather}\label{eq:03191745}
Q_3
=\left(\frac{4}{\delta}|v|\right)^{1/\sigma} 
\end{gather}
so that the condition $|s|p^{-\sigma}<\delta$ is satisfied for $p>Q_3$. 
Since $\RE{(s^2)}=\kappa^2-v^2$ with $s=\kappa+iv$, we obtain
\begin{gather*}
\log{|F_{\sigma,p}(s)|}-\log{F_{\sigma,p}(\kappa)}
=-\frac{1}{4}v^2 p^{-2\sigma}
+O\left(|s|^2p^{-2\sigma-1}+|s|^3p^{-3\sigma}\right) 
\end{gather*}
for $p>Q_3$. 
Note that $|v|>c \kappa$ implies $|s|\asymp|v|$. 
Thus we derive
\begin{gather*}
\log\frac{|F_{\sigma,p}(s)|}{F_{\sigma,p}(\kappa)}
=\left(-\frac{1}{4}+O\left(p^{-1}+v p^{-\sigma}\right)\right)
v^2 p^{-2\sigma}
\leq-\frac{1}{8}v^2 p^{-2\sigma} 
\end{gather*}
for $p>Q_3$ if $\delta>0$ is sufficiently small, and $\kappa=\kappa(\sigma,\tau)>0$ is sufficiently large.  
By the prime number theorem, we have 
\begin{gather*}
\sum_{p>Q_3}
\log{\frac{|F_{\sigma,p}(s)|}{F_{\sigma,p}(\kappa)}}
\leq-\frac{v^2}{8} \sum_{p>Q_3} p^{-2\sigma}
\leq-\frac{v^2}{16(2\sigma-1)} \frac{Q_3^{1-2\sigma}}{\log{Q_3}}. 
\end{gather*}
Inserting \eqref{eq:03191745} to this, we derive the inequality
\begin{gather*}
\prod_{p>Q_3} \frac{|F_{\sigma,p}(s)|}{F_{\sigma,p}(\kappa)}
\leq \exp\left(-c_3(\sigma,c)\frac{|v|^{\frac{1}{\sigma}}}{\log{|v|}}\right)
\end{gather*}
with some positive constant $c_3(\sigma,c)$ if $\kappa>0$ is large. 
Hence the result follows. 
\end{proof}

\begin{proof}[Proof of Theorem \ref{thm:1.1}]
Let $1/2<\sigma \leq1$. 
By the definition of the function $\mathcal{N}_\sigma$, the desired result follows if we have
\begin{gather}\label{eq:03211627}
\mathcal{N}_\sigma(x;\tau)
=\frac{1}{\sqrt{f''_\sigma(\kappa)}}
\left\{ \exp\left(-\frac{x^2}{2f''_\sigma(\kappa)}\right)
+O\left(\kappa^{-\frac{1}{2\sigma}} \sqrt{\log{\kappa}}\right) \right\}. 
\end{gather}
To show this, we use the inverse formula
\begin{gather}\label{eq:03200030}
\mathcal{N}_\sigma(x;\tau)
=\int_{\mathbb{R}} \widetilde{\mathcal{N}}_\sigma(v;\tau) e^{-ixv} \,|dv|
\end{gather}
which is justified by Lemmas \ref{lem:4.4} and \ref{lem:4.5}. 
Note that $\widetilde{\mathcal{N}}_\sigma(v;\tau)$ is represented as 
\begin{gather*}
\widetilde{\mathcal{N}}_\sigma(v;\tau)
=\exp\left(-\frac{f''_\sigma(\kappa)}{2}v^2\right) W(iv)
\end{gather*}
by \eqref{eq:03191630}, where $W$ is the following entire function: 
\begin{gather*}
W(z)
=\exp\left(-\tau z-\frac{f''_\sigma(\kappa)}{2}z^2\right)
\frac{F_\sigma(z+\kappa)}{F_\sigma(\kappa)}. 
\end{gather*}
Put $\lambda=\kappa^{1-\frac{1}{3}\sigma}$. 
Then we have 
\begin{align*}
\int_{-\lambda}^{\lambda} \widetilde{\mathcal{N}}_\sigma(v;\tau) e^{-ixv} \,|dv|
&=\int_{-\lambda}^{\lambda} \exp\left(-\frac{f''_\sigma(\kappa)}{2}v^2\right) e^{-ixv} \,|dv| \\
&\qquad+\int_{-\lambda}^{\lambda} \exp\left(-\frac{f''_\sigma(\kappa)}{2}v^2\right) 
(W(iv)-1) e^{-ixv} \,|dv| \\
&=I_{1,1}+I_{1,2}, 
\end{align*}
say. 
For any $a,b>0$ with $ab^2>1$, we have 
\begin{gather*}
\int_{-\infty}^{\infty} \exp(-av^2) e^{-ixv} \,|dv|
=\frac{1}{\sqrt{2a}} \exp\left(-\frac{x^2}{4a}\right), \\
\int_{b}^{\infty} \exp(-av^2) \,|dv|
\ll \frac{1}{\sqrt{2a}} \exp(-ab^2). 
\end{gather*}
Therefore the first integral $I_{1,1}$ is estimated as
\begin{gather}\label{eq:03200017}
I_{1,1}
=\frac{1}{\sqrt{f''_\sigma(\kappa)}} \left\{\exp\left(-\frac{x^2}{2f''_\sigma(\kappa)}\right)
+O\left(\exp\left(-\frac{f''_\sigma(\kappa)}{2} \lambda^2\right)\right) \right\}
\end{gather}
since $f''_\sigma(\kappa) \lambda^2 \asymp \kappa^{\frac{1}{3\sigma}}(\log{\kappa})^{-1} \to\infty$ as $\kappa \to\infty$ by Propositions \ref{prop:3.1} and \ref{prop:3.2}. 
The second integral $I_{1,2}$ is evaluated as follows. 
We see that $W(z)$ is represented as 
\begin{align*}
W(z)
&=\exp\left(f_\sigma(z+\kappa)-f_\sigma(\kappa)-f'_\sigma(\kappa)z-\frac{f''_\sigma(\kappa)}{2}z^2\right) \\
&=\exp\left(\sum_{j=3}^{\infty} \frac{f^{(j)}_\sigma(\kappa)}{j!}z^j\right)
\end{align*}
at least near the origin. 
Thus we have $W(z)=1+\sum_{j \geq3} w_j z^j/j!$, where 
\begin{gather*}
w_j
=\sum_{k=1}^{\lfloor j/3 \rfloor} \frac{1}{k!}
\sum_{\substack{j_1+\cdots+j_k=j \\ \forall k,~ j_k \geq3}}
\binom{j}{j_1,\ldots,j_k} f^{(j_1)}_\sigma(\kappa) \cdots f^{(j_k)}_\sigma(\kappa). 
\end{gather*}
Additionally, we have the inequality
\begin{align}\label{eq:05040155}
&1+\sum_{j \geq3}\frac{|w_j|}{j!}|z|^j \\
&\leq1+\sum_{j \geq3}\frac{1}{j!} 
\left\{\sum_{k=1}^{\lfloor j/3 \rfloor} \frac{1}{k!}
\sum_{\substack{j_1+\cdots+j_k=j \\ \forall k,~ j_k \geq3}}
\binom{j}{j_1,\ldots,j_k} |f^{(j_1)}_\sigma(\kappa)| \cdots |f^{(j_k)}_\sigma(\kappa)| \right\} |z|^j \nonumber\\
&=\exp\left(\sum_{j=3}^{\infty} \frac{|f^{(j)}_\sigma(\kappa)|}{j!}|z|^j\right) \nonumber
\end{align}
for any $z \in \mathbb{C}$. 
Then, we deduce from Propositions \ref{prop:3.1} and \ref{prop:3.2} the upper bounds
\begin{gather*}
f^{(j)}_\sigma(\kappa)
\ll 2^jj! \frac{\kappa^{\frac{1}{\sigma}-j}}{\log{\kappa}}
\end{gather*}
for all $j \geq3$ by recalling that $\mathfrak{g}_{0,j}(\sigma) \ll j!$ and $\mathfrak{g}_{0,j} \ll j!$. 
By this, we obtain
\begin{gather*}
\sum_{j=3}^{\infty} \frac{|f^{(j)}_\sigma(\kappa)|}{j!}|iv|^j
\ll \frac{\kappa^{\frac{1}{\sigma}}}{\log{\kappa}} \sum_{j=3}^{\infty} \left(\frac{2|v|}{\kappa}\right)^j
\ll \frac{\kappa^{\frac{1}{\sigma}-3}}{\log{\kappa}} |v|^3 
\end{gather*}
for $|v| \leq \lambda$ with the implied constant depending only on $\sigma$. 
Remark that $\kappa^{\frac{1}{\sigma}-3}|v|^3$ is bounded for $|v| \leq \lambda$. 
Hence we obtain from \eqref{eq:05040155} that
\begin{gather*}
\sum_{j \geq3}\frac{|w_j|}{j!}|v|^j
\leq \sum_{n=1}^{\infty} \frac{1}{n!} 
\left(C\frac{\kappa^{\frac{1}{\sigma}-3}}{\log{\kappa}}|v|^3\right)^n
\ll \frac{\kappa^{\frac{1}{\sigma}-3}}{\log{\kappa}}|v|^3, 
\end{gather*}
where $C=C(\sigma)$ is a positive constant. 
As a result, we evaluate $I_{1,2}$ as 
\begin{align*}
|I_{1,2}|
&\leq \int_{-\lambda}^{\lambda} \exp\left(-\frac{f''_\sigma(\kappa)}{2}v^2\right) 
\left(\sum_{j \geq3}\frac{|w_j|}{j!}|v|^j\right) \,|dv| \\
&\ll \frac{\kappa^{\frac{1}{\sigma}-3}}{\log{\kappa}}
\int_{0}^{\infty} \exp\left(-\frac{f''_\sigma(\kappa)}{2}v^2\right) v^3 \,dv \\
&\ll \frac{\kappa^{\frac{1}{\sigma}-3}}{\log{\kappa}} \frac{1}{f''_\sigma(\kappa)^2}. 
\end{align*}
We further note that Propositions \ref{prop:3.1} and \ref{prop:3.2} provide the bound
\begin{gather*}
\frac{\kappa^{\frac{1}{\sigma}-3}}{\log{\kappa}} \frac{1}{\sqrt{f''_\sigma(\kappa)}^3}
\ll \kappa^{-\frac{1}{2\sigma}} \sqrt{\log{\kappa}}. 
\end{gather*}
Therefore, the integral $I_{1,2}$ is evaluated as
\begin{gather}\label{eq:03200018}
I_{1,2}
\ll \frac{1}{\sqrt{f''_\sigma(\kappa)}} \kappa^{-\frac{1}{2\sigma}} \sqrt{\log{\kappa}}. 
\end{gather}
Combining \eqref{eq:03200017} and \eqref{eq:03200018}, we derive
\begin{gather*}
\int_{-\lambda}^{\lambda} \widetilde{\mathcal{N}}_\sigma(v;\tau) e^{-ixv} \,|dv|
=\frac{1}{\sqrt{f''_\sigma(\kappa)}}
\left\{ \exp\left(-\frac{x^2}{2f''_\sigma(\kappa)}\right)
+O\left(\kappa^{-\frac{1}{2\sigma}} \sqrt{\log{\kappa}}\right) \right\}. 
\end{gather*}
Thus, the remaining work is to bound the integral
\begin{gather*}
I_2
=\int_{|v|>\lambda} 
\widetilde{\mathcal{N}}_\sigma(v;\tau) e^{-ixv} \,|dv|. 
\end{gather*}
Let $c_1(\sigma)$, $c_2(\sigma)$, and $c_3(\sigma,c)$ denote the positive constants of Lemmas \ref{lem:4.5} and \ref{lem:4.6}, and put $c_3(\sigma)=c_3(\sigma, c_1(\sigma))$. 
By these lemmas, we have 
\begin{align*}
I_2
&\ll \int_{\lambda}^{c_1(\sigma)\kappa}
\exp\left(-c_2(\sigma)v^2\frac{\kappa^{\frac{1}{\sigma}-2}}{\log{\kappa}}\right) \,dv
+\int_{c_1(\sigma)\kappa}^{\infty}
\exp\left(-c_3(\sigma)\frac{v^{\frac{1}{\sigma}}}{\log{v}}\right) \,dv \\
&\ll \exp\left(-c'_2(\sigma)\frac{\kappa^{\frac{1}{3\sigma}}}{\log{\kappa}}\right)
+\exp\left(-c'_3(\sigma)\frac{\kappa^{\frac{1}{\sigma}}}{\log{\kappa}}\right), 
\end{align*}
where $c'_2(\sigma)$ and $c'_3(\sigma)$ are positive constants. 
Finally, we again use Propositions \ref{prop:3.1} and \ref{prop:3.2} to deduce
\begin{gather*}
I_2
\ll \frac{1}{\sqrt{f''_\sigma(\kappa)}} \kappa^{-\frac{1}{2\sigma}} \sqrt{\log{\kappa}}. 
\end{gather*}
Hence we obtain the result by formula \eqref{eq:03200030}. 
\end{proof}

\subsection{Corollaries}\label{sec:4.3}
As stated in Section \ref{sec:1.1}, one can deduce from Theorem \ref{thm:1.1} several results on the distribution function $\Phi(\sigma,\tau)$ for $1/2<\sigma \leq1$. 

\begin{corollary}\label{cor:4.7}
With the same assumption as in Theorem \ref{thm:1.1}, we obtain
\begin{gather*}
\Phi(\sigma,\tau)
=\frac{F_\sigma(\kappa) e^{-\kappa\tau}}{\kappa \sqrt{2\pi f''_\sigma(\kappa)}}
\left\{ 1+O\left(\kappa^{-\frac{1}{2\sigma}} \sqrt{\log{\kappa}}\right) \right\}
\end{gather*}
if $\tau>0$ is large enough, where the implied constant depends on $\sigma$. 
\end{corollary}

\begin{proof}
By Theorem \ref{thm:1.1}, we calculate $\Phi(\sigma,\tau)$ as
\begin{align*}
\Phi(\sigma,\tau)
&=\int_{0}^{\infty} \mathcal{M}_\sigma(\tau+x) \,|dx| \\
&=\frac{F_\sigma(\kappa) e^{-\kappa \tau}}{\sqrt{f''_\sigma(\kappa)}}
\bigg\{ \int_{0}^{\infty} \exp\left(-\kappa x-\frac{x^2}{2f''_\sigma(\kappa)}\right) \,|dx| \\
&\qquad\qquad\qquad\qquad\qquad 
+O\left(\kappa^{-\frac{1}{2\sigma}} \sqrt{\log{\kappa}} 
\int_{0}^{\infty} e^{-\kappa x} \,|dx| \right) \bigg\} \\
&=\frac{F_\sigma(\kappa) e^{-\kappa \tau}}{\kappa \sqrt{2\pi f''_\sigma(\kappa)}}
\left\{\int_{0}^{\infty} \exp\left(-x-\frac{x^2}{2\kappa^2 f''_\sigma(\kappa)}\right) \,dx 
+O\left(\kappa^{-\frac{1}{2\sigma}} \sqrt{\log{\kappa}}\right)\right\}. 
\end{align*}
Furthermore, we have the asymptotic formula
\begin{align*}
\int_{0}^{\infty} \exp\left(-x-\frac{x^2}{2\kappa^2 f''_\sigma(\kappa)}\right) \,dx
&=1+O\left(\frac{1}{\kappa^2 f''_\sigma(\kappa)}\right) \\
&=1+O\left(\kappa^{-\frac{1}{\sigma}} \log{\kappa}\right)
\end{align*}
by using Propositions \ref{prop:3.1} and \ref{prop:3.2}. 
Hence we obtain the conclusion.  
\end{proof}

We further deduce from Corollary \ref{cor:4.7} the formula
\begin{align*}
\log{\Phi(\sigma,\tau)}
&=f_\sigma(\kappa)-\kappa \tau+O\left(\log{\kappa}+\log{f''_\sigma(\kappa)}\right) \\
&=f_\sigma(\kappa)-\kappa f'_\sigma(\kappa)+O(\log{\kappa})
\end{align*}
for $1/2<\sigma \leq1$ by recalling $\tau=f'_\sigma(\kappa)$ and $f''_\sigma(\kappa) \asymp \kappa^{\frac{1-2\sigma}{\sigma}}(\log{\kappa})^{-1}$. 
By this, we prove Corollaries \ref{cor:1.2} and \ref{cor:1.3} as below. 

\begin{proof}[Proof of Corollary \ref{cor:1.2}]
Let $1/2<\sigma<1$. 
Then it follows from Proposition \ref{prop:3.1} that 
\begin{align}\label{eq:03201422}
\log{\Phi(\sigma,\tau)}
&=-\frac{\kappa^{\frac{1}{\sigma}}}{\log{\kappa}}
\left\{\sum_{n=0}^{N-1} \frac{c_n(\sigma)}{(\log{\kappa})^n}
+O\left( \frac{1}{(\log{\kappa})^N} \right)\right\} \\
&=-c_0(\sigma)\frac{\kappa^{\frac{1}{\sigma}}}{\log{\kappa}}
\left\{1+\sum_{n=1}^{N-1} \frac{c_n(\sigma)}{c_0(\sigma)} \frac{1}{(\log{\kappa})^n}
+O\left( \frac{1}{(\log{\kappa})^N} \right)\right\} \nonumber
\end{align}
holds for any $N \in \mathbb{Z}_{\geq1}$, where we put $c_n(\sigma)=\mathfrak{g}_{n,1}(\sigma)-\mathfrak{g}_{n,0}(\sigma)$. 
Here, we interpret $\sum_{n=1}^{N-1}=0$ if $N=1$. 
Moreover, we have 
\begin{gather*}
\kappa^{\frac{1}{\sigma}}
=B(\sigma)^{\frac{1}{\sigma}} (\tau \log{\tau})^{\frac{1}{1-\sigma}} 
\Bigg\{1+\sum_{n=1}^{N-1} \frac{D_{n,\sigma}(\log_2{\tau})}{(\log{\tau})^n} 
+O\left(\left(\frac{\log_2{\tau}}{\log{\tau}}\right)^{N}\right)\Bigg\}, \\
\log{\kappa}
=\frac{\sigma}{1-\sigma} (\log{\tau})
\left\{1+\sum_{n=1}^{N-1} \frac{D_{n,\sigma}^*(\log_2{\tau})}{(\log{\tau})^n} 
+O\left(\left(\frac{\log_2{\tau}}{\log{\tau}}\right)^{N}\right)\right\}
\end{gather*}
by Proposition \ref{prop:4.2}, where $B(\sigma)$ is determined as \eqref{eq:03201406}, and $D_{n,\sigma}(x)$ and $D_{n,\sigma}^*(x)$ are polynomials of degree at most $n$. 
One can calculate the polynomials when $n=1$ as
\begin{align}
D_{1,\sigma}(x)
&=\frac{1}{\sigma} B_{1,\sigma}(x) \label{eq:03201644} \\ 
&=\frac{1}{1-\sigma}x-\frac{1}{1-\sigma}\log\left(\frac{1-\sigma}{\sigma}\mathfrak{g}_{0,1}(\sigma)\right)
-\frac{1}{\sigma}\frac{\mathfrak{g}_{1,1}(\sigma)}{\mathfrak{g}_{0,1}(\sigma)}, \nonumber \\
D_{1,\sigma}^*(x)
&=x+\frac{1-\sigma}{\sigma} \log{B(\sigma)} \label{eq:03201645} \\
&=x-\log\left(\frac{1-\sigma}{\sigma}\mathfrak{g}_{0,1}(\sigma)\right) \nonumber
\end{align}
with $B_{1,\sigma}(x)$ as in Proposition \ref{prop:4.2}. 
Then, \eqref{eq:03201422} derives 
\begin{gather*}
\log{\Phi(\sigma,\tau)}
=-A(\sigma)\tau^{\frac{1}{1-\sigma}} (\log{\tau})^{\frac{\sigma}{1-\sigma}} 
\left\{ \sum_{n=0}^{N-1} \frac{A_{n,\sigma}(\log_2{\tau})}{(\log{\tau})^n}
+O\left( \left(\frac{\log_2{\tau}}{\log{\tau}}\right)^N \right) \right\},
\end{gather*}
where $A(\sigma)$ is given by $A(\sigma)=c_0(\sigma)B(\sigma)^{\frac{1}{\sigma}} (1-\sigma)/\sigma$, and $A_{n,\sigma}(x)$ are polynomials of degree at most $n$ such that $A_{0,\sigma}(x)=1$. 
Note that $\mathfrak{g}_{0,0}(\sigma)$ and $\mathfrak{g}_{0,1}(\sigma)$ satisfy the relations
\begin{gather*}
\mathfrak{g}_{0,0}(\sigma)
=\frac{1}{\sigma} \mathfrak{g}_{0,1}(\sigma)
\quad\text{and}\quad
\mathfrak{g}_{0,1}(\sigma)
=\int_{0}^{\infty} g(y^{-\sigma}) \,dy
\end{gather*}
by definition.  
Hence $A(\sigma)$ is represented as
\begin{gather*}
A(\sigma)
=c_0(\sigma)B(\sigma)^{\frac{1}{\sigma}} \frac{1-\sigma}{\sigma} 
=(1-\sigma)
\left(\frac{1-\sigma}{\sigma} \mathfrak{g}_{0,1}(\sigma)\right)^{-\frac{1}{1-\sigma}}. 
\end{gather*}
Finally, we calculate the polynomial $A_{1,\sigma}(x)$. 
By formula \eqref{eq:03201422}, we have 
\begin{align*}
\log{\Phi(\sigma,\tau)}
&=-c_0(\sigma) B(\sigma)^{\frac{1}{\sigma}} (\tau \log{\tau})^{\frac{1}{1-\sigma}} 
\Bigg\{1+\frac{D_{1,\sigma}(\log_2{\tau})}{\log{\tau}} 
+O\left(\left(\frac{\log_2{\tau}}{\log{\tau}}\right)^{2}\right)\Bigg\} \\
& \qquad
\times \frac{1-\sigma}{\sigma} (\log{\tau})^{-1}
\Bigg\{1-\frac{D_{1,\sigma}^*(\log_2{\tau})}{\log{\tau}} 
+O\left(\left(\frac{\log_2{\tau}}{\log{\tau}}\right)^{2}\right)\Bigg\} \\
& \qquad\qquad
\times \Bigg\{1+\frac{c_1(\sigma)}{c_0(\sigma)} \frac{1-\sigma}{\sigma} (\log{\tau})^{-1} 
+O\left(\left(\frac{\log_2{\tau}}{\log{\tau}}\right)^{2}\right)\Bigg\} \\
&=A(\sigma)\tau^{\frac{1}{1-\sigma}} (\log{\tau})^{\frac{\sigma}{1-\sigma}} 
\left\{ \frac{A_{1,\sigma}(\log_2{\tau})}{\log{\tau}}
+O\left( \left(\frac{\log_2{\tau}}{\log{\tau}}\right)^2 \right) \right\}, 
\end{align*}
where
\begin{align*}
A_{1,\sigma}(x)
&=D_{1,\sigma}(x)-D_{1,\sigma}^*(x)+\frac{c_1(\sigma)}{c_0(\sigma)} \frac{1-\sigma}{\sigma} \\
&=D_{1,\sigma}(x)-D_{1,\sigma}^*(x)
+\frac{1}{\sigma} \frac{\mathfrak{g}_{1,1}(\sigma)-\mathfrak{g}_{1,0}(\sigma)}{\mathfrak{g}_{0,1}(\sigma)}. 
\end{align*}
Using \eqref{eq:03201644} and \eqref{eq:03201645}, we have
\begin{gather*}
A_{1,\sigma}(x)
=\frac{\sigma}{1-\sigma}x
-\frac{\sigma}{1-\sigma}\log\left(\frac{1-\sigma}{\sigma}\mathfrak{g}_{0,1}(\sigma)\right)
-\frac{1}{\sigma}\frac{\mathfrak{g}_{1,0}(\sigma)}{\mathfrak{g}_{0,1}(\sigma)}. 
\end{gather*}
Thus we obtain the desired representation of $A_{1,\sigma}(x)$ by noting that $\mathfrak{g}_{1,0}(\sigma)$ satisfies
\begin{gather*}
\mathfrak{g}_{1,0}(\sigma)
=-\sigma \int_{0}^{\infty} g(y^{-\sigma}) \log{y} \,dy. 
\end{gather*}
\end{proof}

\begin{proof}[Proof of Corollary \ref{cor:1.2}]
In this case, we apply Proposition \ref{prop:3.2} to deduce
\begin{align}\label{eq:03241158}
\log{\Phi(1,\tau)}
&=-\frac{\kappa}{\log{\kappa}}
\left\{\sum_{n=0}^{N-1} \frac{c_n}{(\log{\kappa})^n}
+O\left( \frac{1}{(\log{\kappa})^N} \right)\right\} \\
&=-\frac{2\kappa}{\log{\kappa}}
\left\{1+\sum_{n=1}^{N-1} \frac{c_n}{2} \frac{1}{(\log{\kappa})^n}
+O\left( \frac{1}{(\log{\kappa})^N} \right)\right\}, \nonumber
\end{align}
where $c_n=\mathfrak{g}_{n,1}-\mathfrak{g}_{n,0}$. 
Here, we remark that $c_0=2$ since we have 
\begin{gather*}
\mathfrak{g}_{0,1}
=\lim_{\epsilon \to +0} 
\left(\left[ \frac{g_*(u)}{u}\right]_0^{1-\epsilon}
+\left[ \frac{g_*(u)}{u}\right]_{1+\epsilon}^{\infty}\right)
+\int_{0}^{\infty} \frac{g_*(u)}{u^2} \,du
=2+\mathfrak{g}_{0,0}
\end{gather*}
by integrating by parts. 
Furthermore, we deduce from Proposition \ref{prop:4.3} and \eqref{eq:03201702} the asymptotic formulas
\begin{gather*}
\kappa
=\exp\left(t-\frac{1}{2}\mathfrak{g}_{0,1}\right) 
\left\{1+\frac{b_1}{t}+\sum_{n=2}^{N-1} \frac{b_n}{t^n}
+O\left(\frac{1}{t^N}\right) \right\}, \\
\log{\kappa}
=t \left\{1-\frac{\mathfrak{g}_{0,1}}{2t}
-\sum_{n=2}^{N-1} \frac{\beta_{n-1}}{t^n}
+O\left(\frac{1}{t^N}\right)\right\}, 
\end{gather*}
where $b_n$ and $\beta_n$ are real numbers such that $b_1=-\beta_1=-\frac{1}{8}\mathfrak{g}_{0,1}^2-\frac{1}{2}\mathfrak{g}_{1,1}$. 
Using these formulas, we obtain
\begin{align*}
\log{\Phi(1,\tau)}
=-\frac{e^{t-A}}{t}
\left\{ \sum_{n=0}^{N-1} \frac{a_n}{t^n}
+O\left(\frac{1}{t^N}\right) \right\}
\end{align*}
by \eqref{eq:03241158}, where $A$ is given by
\begin{gather*}
A
=\frac{1}{2}\mathfrak{g}_{0,1}-\log{2}
=1+\frac{1}{2}\mathfrak{g}_{0,0}-\log{2}, 
\end{gather*}
and $a_n$ are real numbers such that $a_0=1$ and
\begin{gather*}
a_1
=b_1+\frac{1}{2}\mathfrak{g}_{0,1}+\frac{1}{2}c_1
=-\frac{1}{8}\mathfrak{g}_{0,1}^2
-\frac{1}{2}\mathfrak{g}_{1,0}
+\frac{1}{2}. 
\end{gather*}
Finally, we see that $\mathfrak{g}_{0,1}$ and $\mathfrak{g}_{1,0}$ are represented as
\begin{gather*}
\mathfrak{g}_{0,1}
=\int_{0}^{\infty} g_*(y^{-1}) \,dy
\quad\text{and}\quad
\mathfrak{g}_{1,0}
=-\int_{0}^{\infty} g_*(y^{-1}) \log{y} \,dy. 
\end{gather*}
Hence we complete the proof. 
\end{proof}

\begin{remark}\label{rem:4.8}
The constant $A$ of this paper is consistent with the constant $A_k$ of Lamzouri \cite[Theorem 0.2]{Lamzouri2010} when $k=1$, while he represented it in a slightly different form. 
Especially, we have 
\begin{gather*}
A_1
=1+\int_{0}^{\infty} \frac{h_*(u)}{u^2} \,du
\end{gather*}
according to the representation by Lamzouri, where $h_*(u)$ is defined as
\begin{gather*}
h_*(u)
=
\begin{cases}
h(u)
& \text{if $0<u<1$}, 
\\
h(u)-u
& \text{if $u \geq 1$} 
\end{cases}
\end{gather*}
by using the cumulant-generating function
\begin{gather*}
h(u)
=\log \left(\frac{2}{\pi}\int_{0}^{\pi} \exp(u\cos \theta) \sin^2 \theta \,d \theta\right). 
\end{gather*}
Then we have $h(2u)=g(u)$, and therefore, we see that 
\begin{align*}
A_1
&=1+\frac{1}{2}\int_{0}^{\infty} \frac{h_*(2u)}{u^2} \,du \\
&=1+\frac{1}{2}\int_{0}^{\infty} \frac{g_*(u)}{u^2} \,du
+\frac{1}{2}\int_{1/2}^{1} \frac{h_*(2u)-g_*(u)}{u^2} \,du\\
&=1+\frac{1}{2}\int_{0}^{\infty} \frac{g_*(u)}{u^2} \,du-\log{2}
\end{align*}
which equals to $A$ of this paper. 
We further remark that $A$ should be consistent with the constant $\gamma_0$ of Liu--Royer--Wu \cite[Corollary 1.5]{LiuRoyerWu2008}. 
However, it appears that they miscalculated the value of $\gamma_0$ by forgetting the term $-\log{2}$. 
\end{remark}

\section{Comparisons of distribution functions}\label{sec:5}
In this section, we prove Theorem \ref{thm:1.4} and its corollaries. 
For this, we apply the following asymptotic formula on the complex moments of $L(\sigma,f)$ which was obtained in the previous paper of the author \cite{Mine2020+}. 

\begin{proposition}\label{prop:5.1}
Let $1/2<\sigma \leq1$ and $B \geq1$. 
Then there exist positive constants $a=a(\sigma,B)$, $b=b(\sigma,B)$ and a subset $\mathcal{E}_q=\mathcal{E}_q(\sigma,B)$ of $B_2(q)$ such that 
\begin{gather}\label{eq:03210326}
\frac{1}{\# B_2(q)} \sum_{f \in B_2(q) \setminus \mathcal{E}_q} L(\sigma,f)^s
=F_\sigma(s)
+O\left(\frac{F_\sigma(\kappa)}{(\log{q})^{B+2}}\right)
\end{gather}
holds uniformly for $s=\kappa+iv \in \mathbb{C}$ with $|s| \leq a R_\sigma(q)$, where we define
\begin{gather*}
R_\sigma(q) 
=
\begin{cases}
(\log{q})^\sigma &{\text{if $1/2<\sigma<1$}}, \\ 
(\log{q})(\log_2{q}\,\log_3{q})^{-1} &{\text{if $\sigma=1$}}.
\end{cases}
\end{gather*}
Furthermore, we have 
\begin{gather}\label{eq:03210327}
\# \mathcal{E}_q
\ll q\exp\left(-b \frac{\log{q}}{\log_2{q}}\right). 
\end{gather}
The implied constants in \eqref{eq:03210326} and \eqref{eq:03210327} depend on $\sigma$ and $B$. 
\end{proposition}

Note that similar results were obtained by Cogdell--Michel \cite{CogdellMichel2004} and Lamzouri \cite{Lamzouri2011b} for the averages weighted by $\omega_f=(4\pi \langle f,f \rangle)^{-1}$. 
The results were applied to study the weighted distribution functions $\widetilde{\Phi}_q(\sigma,\tau)$ and $\widetilde{\Phi}(\sigma,\tau)$ in \cite{Lamzouri2010, Lamzouri2011b}. 
However, the method of this paper is different from theirs in terms of our using the following Esseen inequality. 

\begin{lemma}[Esseen inequality \cite{Loeve1977}]\label{lem:5.2}
Let $P$ and $Q$ be two probability measures on $(\mathbb{R}, \mathcal{B}(\mathbb{R}))$ with distribution functions $\Phi(\xi)=P((-\infty,\xi])$ and $\Psi(\xi)=Q((-\infty,\xi])$, respectively.
Assume that $\Psi$ is differentiable, and that $K=\sup_{\xi \in \mathbb{R}} |\Psi'(\xi)|$ is finite. 
Then we have 
\begin{gather}\label{eq:03211707}
\sup_{\xi \in \mathbb{R}} |\Phi(\xi)-\Psi(\xi)|
\leq \frac{2}{\pi} \int_{0}^{R} \frac{|\phi(v)-\psi(v)|}{v} \,dv
+\frac{24K}{\pi}R^{-1}
\end{gather}
for every $R>0$, where $\phi$ and $\psi$ are the characteristic functions defined as
\begin{gather*}
\phi(v)
=\int_{\mathbb{R}} e^{ivx} \,dP(x)
\quad\text{and}\quad
\psi(v)
=\int_{\mathbb{R}} e^{ivx} \,dQ(x). 
\end{gather*}
\end{lemma}

Let $1/2<\sigma \leq1$ and $B \geq1$. 
Define the functions $U$ and $V$ on $\mathbb{R}$ as
\begin{align*}
U(\xi)
&=\frac{\# \left\{f \in B_2(q) \setminus \mathcal{E}_q ~\middle|~ \log{L}(\sigma,f)\leq \xi \right\}}
{\# B_2(q) \setminus \mathcal{E}_q}, \\
V(\xi)
&=\int_{-\infty}^{\xi} \mathcal{M}_\sigma(x) \,|dx|, 
\end{align*}
where $\mathcal{E}_q=\mathcal{E}_q(\sigma,B)$ is the subset of $B_2(q)$ as in Proposition \ref{prop:5.1}. 
Then we apply Lemma \ref{lem:5.2} with the probability measures defined as
\begin{gather*}
P(A)
=\frac{\displaystyle{\int_\mathbb{R} 1_{A+\tau}(\xi) e^{\kappa \xi} \,dU(\xi)}}
{\displaystyle{\int_\mathbb{R} {e}^{\kappa \xi} \,dU(\xi)}}
\quad\text{and}\quad
Q(A)
=\frac{\displaystyle{\int_\mathbb{R} 1_{A+\tau}(\xi) e^{\kappa \xi} \,dV(\xi)}}
{\displaystyle{\int_\mathbb{R} {e}^{\kappa \xi} \,dV(\xi)}}
\end{gather*}
for $A\in\mathcal{B}(\mathbb{R})$, where $\kappa=\kappa(\sigma,\tau)$ is the solution to equation \eqref{eq:03031628}. 
Here, we denote by $A+\tau$ the set $\{a+\tau \mid a \in A\}$, and $1_{S}$ is the indicator function of a set $S \subset \mathbb{R}$. 
We further put $\Phi(\xi)=P((-\infty,\xi])$ and $\Psi(\xi)=Q((-\infty,\xi])$ as above. 

\begin{lemma}\label{lem:5.3}
With the notation above, we have
\begin{gather*}
\sup_{\xi \in \mathbb{R}} |\Phi(\xi)-\Psi(\xi)|
\ll \frac{1}{(\log{q})^{B+1}}
+\frac{1}{R_\sigma(q) \sqrt{f''_\sigma(\kappa)}} 
\end{gather*}
if the condition $\kappa \leq aR_\sigma(q)/2$ is satisfied with the positive constant $a=a(\sigma,B)$ of Proposition \ref{prop:5.1}. 
Here, the implied constant depends on $\sigma$ and $B$. 
\end{lemma}

\begin{proof}
First, we check the assumption on $\Psi$ of Lemma \ref{lem:5.2}. 
Note that the identity
\begin{gather*}
Q(A)
=\frac{1}{F_\sigma(\kappa)} \int_{A+\tau} e^{\kappa x} \mathcal{M}_\sigma(x) \,|dx|
=\int_{A} \mathcal{N}_\sigma(x;\tau) \,|dx|
\end{gather*}
holds for the function $\mathcal{N}_\sigma$ of \eqref{eq:03211608}. 
Thus $\Psi(\xi)=\int_{-\infty}^{\xi} \mathcal{N}_\sigma(x;\tau) \,|dx|$ is differentiable, and we have 
\begin{gather}\label{eq:03212144}
K
=\sup_{\xi \in \mathbb{R}} |\Psi'(\xi)|
=\sup_{\xi \in \mathbb{R}} \frac{\mathcal{N}_\sigma(\xi;\tau)}{\sqrt{2\pi}}
\ll \frac{1}{\sqrt{f''_\sigma(\kappa)}}
<\infty
\end{gather}
by asymptotic formula \eqref{eq:03211627}. 
Hence Lemma \ref{lem:5.2} is available for the probability measures $P$ and $Q$. 
Next, we have the formula
\begin{align*}
&\int_\mathbb{R} {e}^{(\kappa+iv)\xi} \,dU(\xi) \\
&=\frac{1}{\# B_2(q) \setminus \mathcal{E}_q} \sum_{f \in B_2(q) \setminus \mathcal{E}_q} L(\sigma,f)^{\kappa+iv} \\
&=\frac{1}{\# B_2(q)} \sum_{f \in B_2(q) \setminus \mathcal{E}_q} L(\sigma,f)^{\kappa+iv}
+O\left(\exp\left(-b \frac{\log{q}}{\log_2{q}}\right) F_\sigma(\kappa) \right) 
\end{align*}
by applying \eqref{eq:03210327}. 
Furthermore, the equality
\begin{gather*}
\int_\mathbb{R} {e}^{(\kappa+iv)\xi} \,dV(\xi)
=F(\kappa+iv) 
\end{gather*}
is valid by definition. 
With the above preparations, we determine the parameter $R>0$ as $R=aR_\sigma(q)/2$ so that $|\kappa+iv|\leq a R_\sigma(q)$ is satisfied for $0<v<R$. 
Then Proposition \ref{prop:5.1} yields
\begin{gather}\label{eq:03211714}
\int_\mathbb{R} {e}^{(\kappa+iv)\xi} \,dU(\xi)
=\int_\mathbb{R} {e}^{(\kappa+iv)\xi} \,dV(\xi)
+O\left(\frac{F_\sigma(\kappa)}{(\log{q})^{B+2}}\right)
\end{gather}
for $0<v<R$. 
In addition, we see similarly that
\begin{gather}\label{eq:03211715}
\int_\mathbb{R} {e}^{\kappa \xi} \,dU(\xi)
\asymp \int_\mathbb{R} {e}^{\kappa \xi} \,dV(\xi)
=F_\sigma(\kappa) 
\end{gather}
holds. 
Using these formulas, we evaluate the integral of the right-hand side of \eqref{eq:03211707} as follows. 
The characteristic functions $\phi$ and $\psi$ are calculated as
\begin{gather}\label{eq:05111749}
\phi(v)
=\frac{\displaystyle{\int_\mathbb{R} {e}^{(\kappa+iv) \xi} \,dU(\xi)}}
{\displaystyle{\int_\mathbb{R}{e}^{\kappa \xi} \,dU(\xi)}} e^{-iv\tau}
\quad\text{and}\quad
\psi(v)
=\frac{\displaystyle{\int_\mathbb{R} {e}^{(\kappa+iv) \xi} \,dV(\xi)}}
{\displaystyle{\int_\mathbb{R} {e}^{\kappa \xi} \,dV(\xi)}} e^{-iv\tau}. 
\end{gather}
Hence we obtain
\begin{align*}
|\phi(v)-\psi(v)|
&\leq \frac{\displaystyle{\left|\int_\mathbb{R} {e}^{(\kappa+iv) \xi} \,d(U-V)(\xi)\right|}}
{\displaystyle{\int_\mathbb{R} {e}^{\kappa \xi} \,dV(\xi)}}
+\frac{\displaystyle{\left|\int_\mathbb{R} {e}^{\kappa \xi} \,d(U-V)(\xi)\right|}}
{\displaystyle{\int_\mathbb{R} {e}^{\kappa \xi} \,dV(\xi)}}\\
&\ll(\log{q})^{-B-2}
\end{align*}
for $0<v<R$ by using \eqref{eq:03211714} and \eqref{eq:03211715}. 
Put $r=\exp\left(-L\log{q}/\log_2{q}\right)$ with a constant $L=L(\sigma,B)>0$ chosen later. 
Then we have 
\begin{align}\label{eq:03212145}
\int_{r}^{R} \frac{|\phi(v)-\psi(v)|}{v} \,du
&\ll \left(\log{\frac{R}{r}}\right) (\log{q})^{-B-2} \\
&\ll (\log{q})^{-B-1}. \nonumber
\end{align}
For $0<v \leq r$, we estimate $\phi(v)$ and $\psi(v)$ by using the formula $e^{i \theta}=1+O(|\theta|)$ with arbitrary $\theta \in \mathbb{R}$. 
We have 
\begin{gather*}
\int_\mathbb{R} {e}^{(\kappa+iv) \xi} \,dU(\xi)
=\int_\mathbb{R} {e}^{\kappa \xi} \,dU(\xi)
+O\left(v \int_\mathbb{R} |\xi| {e}^{\kappa \xi} \,dU(\xi)\right). 
\end{gather*}
By the Cauchy--Schwarz inequality and \eqref{eq:03211715}, we obtain
\begin{gather*}
\int_\mathbb{R} |\xi| {e}^{\kappa \xi} \,dU(\xi)
\ll \sqrt{M_q} \sqrt{F_\sigma(2\kappa)},  
\end{gather*}
where we put
\begin{gather*}
M_q
=\int_\mathbb{R} |\xi|^2 \,dU(\xi)
=\frac{1}{\# B_2(q) \setminus \mathcal{E}_q} 
\sum_{f \in B_2(q) \setminus \mathcal{E}_q} |\log{L}(\sigma,f)|^2. 
\end{gather*}
Thus, $\phi$ is estimated as
\begin{gather}\label{eq:05111750}
\phi(v)
=e^{-iv \tau} \left(1+O\left(v \sqrt{M_q} \frac{\sqrt{F_\sigma(2\kappa)}}{F_\sigma(\kappa)}\right)\right) 
\end{gather}
by recalling \eqref{eq:05111749}. 
In a similar way, we obtain the formula
\begin{gather}\label{eq:05111751}
\psi(v)
=e^{-iv \tau} \left(1+O\left(v \sqrt{M} \frac{\sqrt{F_\sigma(2\kappa)}}{F_\sigma(\kappa)}\right)\right), 
\end{gather}
where $M$ is the constant represented as
\begin{gather*}
M
=\int_\mathbb{R} |\xi|^2 \,dV(\xi)
=\int_{\mathbb{R}} |x|^2 \mathcal{M}_\sigma(x) \,|dx|.
\end{gather*}
One can deduce from Proposition \ref{prop:5.1} the estimate $M_q-M \ll F_\sigma(\kappa)(\log{q})^{-B}$ by differentiating both sides of \eqref{eq:03210326} in $s$. 
Additionally, we use Propositions \ref{prop:3.1} and \ref{prop:3.2} to derive the bounds
\begin{align*}
\sqrt{F_\sigma(2\kappa)}
=\exp\left(\frac{1}{2}f_\sigma(2\kappa)\right) 
\leq
\begin{cases}
\displaystyle{\exp\left(L_1\frac{\kappa^{\frac{1}{\sigma}}}{\log{\kappa}}\right) }
& \text{if $1/2<\sigma<1$}, 
\\
\displaystyle{\exp\left(L_1\log_2{\kappa}\right) }
& \text{if $\sigma=1$}, 
\end{cases}
\end{align*}
where $L_1=L_1(\sigma)$ is a positive constant. 
Since $\kappa \leq aR_\sigma(q)/2$, there exists a positive constant $L_2=L_2(\sigma,B)$ such that
\begin{gather*}
\sqrt{F_\sigma(2\kappa)}
\leq \exp\left(L_2 \frac{\log{q}}{\log_2{q}}\right)
\end{gather*}
in both cases $1/2<\sigma<1$ and $\sigma=1$.
Hence, we evaluate the difference between $\phi$ and $\psi$ as 
\begin{gather*}
|\phi(v)-\psi(v)|
\ll v \left|\sqrt{M_q}-\sqrt{M}\right| \frac{\sqrt{F_\sigma(2\kappa)}}{F_\sigma(\kappa)}
\ll v \exp\left(L_2 \frac{\log{q}}{\log_2{q}}\right)
\end{gather*}
by \eqref{eq:05111750} and \eqref{eq:05111751}. 
Then, we choose the constant $L>0$ in the definition of $r$ as $L=2L_2$. 
We have 
\begin{gather}\label{eq:03212146}
\int_{0}^{r} \frac{|\phi(v)-\psi(v)|}{v} \,du
\ll \exp\left(-L_2 \frac{\log{q}}{\log_2{q}}\right). 
\end{gather}
Combining \eqref{eq:03212144}, \eqref{eq:03212145}, and \eqref{eq:03212146}, we deduce from Lemma \ref{lem:5.2} the desired result.
\end{proof}

\begin{proof}[Proof of Theorem \ref{thm:1.4}]
For $1/2<\sigma \leq1$, we define 
\begin{gather*}
\Phi_q^*(\sigma,\tau)
=\frac{\# \left\{f \in B_2(q) \setminus \mathcal{E}_q ~\middle|~ \log{L}(\sigma,f)>\tau \right\}}
{\# B_2(q) \setminus \mathcal{E}_q}, 
\end{gather*}
where $\mathcal{E}_q=\mathcal{E}_q(\sigma,B)$ is the subset of $B_2(q)$ as in Proposition \ref{prop:5.1}. 
Then the identities
\begin{align*}
\Phi_q^*(\sigma,\tau)
&=e^{-\tau\kappa} \int_\mathbb{R} {e}^{\kappa \xi} \,dU(\xi)
\int_0^\infty {e}^{-\kappa \xi} \,d\Phi(\xi), \\
\Phi(\sigma,\tau)
&=e^{-\tau\kappa} \int_\mathbb{R} {e}^{\kappa \xi}\,dV(\xi)
\int_0^\infty {e}^{-\kappa \xi} \,d\Psi(\xi) 
\end{align*}
hold by the definitions of the probability measures $P$ and $Q$. 
Therefore we have 
\begin{align}\label{eq:03212214}
\left|\Phi_q^*(\sigma,\tau)-\Phi(\sigma,\tau)\right|
&\leq e^{-\tau\kappa} \int_0^\infty {e}^{-\kappa \xi} \,d\Psi(\xi) 
\left|\int_\mathbb{R} {e}^{\kappa \xi} \,d(U-V)(\xi)\right| \\
&\qquad
+e^{-\tau\kappa} \int_\mathbb{R} {e}^{\kappa \xi} \,dU(\xi)
\left|\int_0^\infty {e}^{-\kappa \xi} \,d(\Phi-\Psi)(\xi)\right|. \nonumber
\end{align}
Suppose that the condition $\kappa \leq aR_\sigma(q)/2$ is satisfied. 
Since $\Psi(\xi)$ is represented as $\Psi(\xi)=\int_{-\infty}^{\xi} \mathcal{N}_\sigma(x;\tau) \,|dx|$, the upper bound
\begin{gather*}
\int_0^\infty {e}^{-\kappa \xi} \,d\Psi(\xi)
=\int_0^\infty {e}^{-\kappa \xi} \mathcal{N}_\sigma(x;\tau) \,|dx|
\ll \frac{1}{\kappa \sqrt{f''_\sigma(\kappa)}}
\end{gather*}
follows from asymptotic formula \eqref{eq:03211627}. 
Using \eqref{eq:03211714}, we further deduce 
\begin{gather*}
\int_\mathbb{R} {e}^{\kappa \xi} \,d(U-V)(\xi)
\ll \frac{F_\sigma(\kappa)}{(\log{q})^{B}}. 
\end{gather*}
Thus the first term of \eqref{eq:03212214} is estimated as 
\begin{gather}\label{eq:03212301}
e^{-\tau\kappa} \int_0^\infty {e}^{-\kappa \xi} \,d\Psi(\xi) 
\left|\int_\mathbb{R} {e}^{\kappa \xi} \,d(U-V)(\xi)\right|
\ll \frac{\Phi(\sigma,\tau)}{(\log{q})^B} 
\end{gather}
by Corollary \ref{cor:4.7}. 
Next, we estimate the second term of \eqref{eq:03212214} by applying \eqref{eq:03211715} and Lemma \ref{lem:5.3}. 
We have 
\begin{align*}
\int_0^\infty {e}^{-\kappa \xi} \,d(\Phi-\Psi)(\xi)
&\ll \sup_{\xi \in \mathbb{R}} |\Phi(\xi)-\Psi(\xi)| \\
&\ll \frac{1}{(\log{q})^{B+1}}
+\frac{1}{R_\sigma(q) \sqrt{f''_\sigma(\kappa)}} 
\end{align*}
by the integration by parts. 
Hence we derive
\begin{gather*}
e^{-\tau\kappa} \int_\mathbb{R} {e}^{\kappa \xi} \,dU(\xi)
\left|\int_0^\infty {e}^{-\kappa \xi} \,d(\Phi-\Psi)(\xi)\right| 
\ll \Phi(\sigma,\tau)
\left(\frac{\kappa \sqrt{f''_\sigma(\kappa)}}{(\log{q})^{B+1}}+\frac{\kappa}{R_\sigma(q)}\right) 
\end{gather*}
by using Corollary \ref{cor:4.7} again. 
Furthermore, in both cases $1/2<\sigma<1$ and $\sigma=1$, the estimate $\kappa \sqrt{f''_\sigma(\kappa)} \ll \log{q}$ follows from Propositions \ref{prop:3.1} and \ref{prop:3.2} since we suppose that $\kappa$ satisfies $\kappa \leq aR_\sigma(q)/2$.  
It yields
\begin{gather}\label{eq:03212302}
e^{-\tau\kappa} \int_\mathbb{R} {e}^{\kappa \xi} \,dU(\xi)
\left|\int_0^\infty {e}^{-\kappa \xi} \,d(\Phi-\Psi)(\xi)\right| 
\ll \Phi(\sigma,\tau)
\left(\frac{1}{(\log{q})^{B}}+\frac{\kappa}{R_\sigma(q)}\right). 
\end{gather}
By \eqref{eq:03212301} and \eqref{eq:03212302}, we obtain 
\begin{gather*}
\Phi_q^*(\sigma,\tau)-\Phi(\sigma,\tau)
\ll \Phi(\sigma,\tau)
\left(\frac{1}{(\log{q})^{B}}+\frac{\kappa}{R_\sigma(q)}\right).
\end{gather*}
The difference between $\Phi_q(\sigma,\tau)$ and $\Phi_q^*(\sigma,\tau)$ can be evaluated as
\begin{gather*}
\Phi_q(\sigma,\tau)-\Phi_q^*(\sigma,\tau)
\ll \exp\left(-b \frac{\log{q}}{\log_2{q}}\right) 
\end{gather*}
by using \eqref{eq:03210327}. 
Hence we arrive at the formula
\begin{gather}\label{eq:03212324}
\Phi_q(\sigma,\tau)
=\Phi(\sigma,\tau)
\left(1+O\left(\frac{1}{(\log{q})^{B}}+\frac{\kappa}{R_\sigma(q)}\right)\right)
+O\left(\exp\left(-b \frac{\log{q}}{\log_2{q}}\right)\right) 
\end{gather}
for $1/2<\sigma \leq1$. 
In the case $1/2<\sigma<1$, we recall that $\kappa$ satisfies
\begin{gather*}
\kappa \asymp (\tau \log{\tau})^{\frac{\sigma}{1-\sigma}}
\end{gather*}
by Proposition \ref{prop:4.2}. 
Thus one can take a small positive constant $c(\sigma,B)$ so that the condition $\kappa \leq aR_\sigma(q)/2$ holds for $1\ll \tau \leq c(\sigma,B) (\log{q})^{1-\sigma}(\log_2{q})^{-1}$.
Furthermore, Corollary \ref{cor:1.2} yields
\begin{gather*}
\Phi(\sigma,\tau)
\gg \exp\left(-\frac{b}{2} \frac{\log{q}}{\log_2{q}}\right)
\end{gather*}
in the range $1\ll \tau \leq c(\sigma,B) (\log{q})^{1-\sigma}(\log_2{q})^{-1}$ if $c(\sigma,B)$ is sufficiently small. 
Thus, the desired conclusion
\begin{gather*}
\Phi_q(\sigma,\tau)
=\Phi(\sigma,\tau)
\left(1+O\left(\frac{1}{(\log{q})^{B}}+\frac{(\tau \log{\tau})^{\frac{\sigma}{1-\sigma}}}{(\log{q})^\sigma}\right)\right) 
\end{gather*}
follows from \eqref{eq:03212324}. 
If $\sigma=1$, then we have $\kappa \asymp e^t$ by Proposition \ref{prop:4.3}, where we put $\tau=2\log{t}+2\gamma$. 
Hence there exists a large positive constant $c(B)$ such that $\kappa \leq aR_1(q)/2$ holds for $1\ll t \leq \log_2{q}-\log_3{q}-\log_4{q}-c(B)$. 
In this case, we have the lower bound
\begin{gather*}
\Phi(1,\tau)
\gg \exp\left(-\frac{b}{2} \frac{\log{q}}{\log_2{q}}\right)
\end{gather*}
by Corollary \ref{cor:1.3} in the range $1\ll t \leq \log_2{q}-\log_3{q}-\log_4{q}-c(B)$ with $c(B)$ large enough. 
Therefore \eqref{eq:03212324} yields
\begin{gather*}
\Phi_q(1,\tau)
=\Phi(1,\tau)
\left( 1+O\left(\frac{1}{(\log{q})^B}
+\frac{e^t}{(\log{q})(\log_2{q}\,\log_3{q})^{-1}}\right) \right)
\end{gather*}
as desired. 
\end{proof}

\begin{proof}[Proof of Corollaries \ref{cor:1.5} and \ref{cor:1.6}]
If we put $\Phi_q(\sigma,\tau)=\Phi(\sigma,\tau)\left(1+E_q(\sigma,\tau)\right)$, then Theorem \ref{thm:1.4} $(\mathrm{i})$ deduces the bound
\begin{gather}\label{eq:03232221}
E_q(\sigma,\tau)
\ll \frac{1}{\log{q}}+\frac{(\tau \log{\tau})^{\frac{\sigma}{1-\sigma}}}{(\log{q})^\sigma}. 
\end{gather}
for $1/2<\sigma<1$ in the range $1\ll \tau \leq c(\sigma,1) (\log{q})^{1-\sigma}(\log_2{q})^{-1}$. 
We further apply Corollary \ref{cor:1.2} to obtain
\begin{align*}
\log{\Phi_q(\sigma,\tau)} 
&=\log{\Phi(\sigma,\tau)}+O\left(|E_q(\sigma,\tau)|\right) \\
&=-A(\sigma)\tau^{\frac{1}{1-\sigma}} (\log{\tau})^{\frac{\sigma}{1-\sigma}} 
\Bigg\{ \sum_{n=0}^{N-1} \frac{A_{n,\sigma}(\log_2{\tau})}{(\log{\tau})^n}+ \\
&\qquad\qquad
+O\left( \left(\frac{\log_2{\tau}}{\log{\tau}}\right)^N 
+\tau^{-\frac{1}{1-\sigma}} (\log{\tau})^{-\frac{\sigma}{1-\sigma}} |E_q(\sigma,\tau)| \right) \Bigg\}. 
\end{align*}
It completes the proof of Corollary \ref{cor:1.5} since we obtain
\begin{gather*}
\tau^{-\frac{1}{1-\sigma}} (\log{\tau})^{-\frac{\sigma}{1-\sigma}} |E_q(\sigma,\tau)|
\ll \left(\frac{\log_2{\tau}}{\log{\tau}}\right)^N
\end{gather*}
in the range $1\ll \tau \leq c(\sigma,1) (\log{q})^{1-\sigma}(\log_2{q})^{-1}$  by \eqref{eq:03232221}. 
Corollary \ref{cor:1.6} can be proved similarly. 
\end{proof}


\providecommand{\bysame}{\leavevmode\hbox to3em{\hrulefill}\thinspace}
\providecommand{\MR}{\relax\ifhmode\unskip\space\fi MR }
\providecommand{\MRhref}[2]{%
  \href{http://www.ams.org/mathscinet-getitem?mr=#1}{#2}
}
\providecommand{\href}[2]{#2}

\end{document}